\documentclass[11pt]{amsart}
\usepackage{geometry}                
\usepackage{graphicx}
\usepackage{amssymb}
\usepackage{epstopdf}
\usepackage{amsmath}
\usepackage{color}
\usepackage{pdfsync}
 \usepackage{bbm}
\usepackage[all]{xy}
\xyoption{matrix}
\xyoption{arrow}

\newtheorem{thm}{Theorem}[subsection]
\newtheorem{lem}[thm]{Lemma}
\newtheorem{cor}[thm]{Corollary}
\newtheorem{prop}[thm]{Proposition}

\theoremstyle{definition}
\newtheorem{defn}[thm]{Definition}

\theoremstyle{remark}
\newtheorem{rem}[thm]{Remark}
 
\numberwithin{equation}{section}
  
\def\ccc{{\mathcal F}}
 
\def\vs#1{\vskip .#1 cm} 

\def\xrarrow{\xrightarrow} 

\def\ot{\leftarrow}

\def\-{\text{-}}
 
\def\noteq{\neq}

\def\<{\left<}
\def\>{\right>}

 \newcommand{\into}{\hookrightarrow}
 \newcommand{\onto}{\twoheadrightarrow}
 \newcommand{\cof}{\rightarrowtail}

\DeclareMathOperator{\Hom}{Hom}%
\DeclareMathOperator{\Ext}{Ext}%
\DeclareMathOperator{\End}{End}%
\DeclareMathOperator{\Aut}{Aut}

\newcommand{\field}[1]{\mathbb{#1}}
\newcommand{\ZZ}{\ensuremath{{\field{Z}}}}

\newcommand{\RR}{\ensuremath{{\field{R}}}}

\newcommand{\NN}{\ensuremath{{\field{N}}}}

\newcommand{\kk}{\ensuremath{{\mathbbm{k}}}}

\newcommand{\sa}{\ensuremath{^{sa}}}  
\newcommand{\double}{\ensuremath{^{(2)}}}  
\newcommand{\mat}[1]{\left[\begin{matrix}#1
\end{matrix}\right]}
 
\newcommand{\commentout}[1]{}

\def\el{\ell}
\def\ll{\lambda}
\def\LL{\Lambda}
\newcommand{\cA}{\ensuremath{{\mathcal{A}}}}

\newcommand{\cB}{\ensuremath{{\mathcal{B}}}}
\newcommand{\cC}{\ensuremath{{\mathcal{C}}}}
\newcommand{\cD}{\ensuremath{{\mathcal{D}}}}
\newcommand{\cE}{\ensuremath{{\mathcal{E}}}}
\newcommand{\cF}{\ensuremath{{\mathcal{F}}}}

\newcommand{\cJ}{\ensuremath{{\mathcal{J}}}}

\newcommand{\cL}{\ensuremath{{\mathcal{L}}}}
\newcommand{\cM}{\ensuremath{{\mathcal{M}}}}

\newcommand{\cP}{\ensuremath{{\mathcal{P}}}}

\newcommand{\cT}{\ensuremath{{\mathcal{T}}}}

\newcommand{\cX}{\ensuremath{{\mathcal{X}}}}

\def\a{\alpha}
\def\b{\beta}
\def\g{\gamma}

\def\d{\partial}

\def\e{\epsilon}
\def\f{\varphi}

\def\s{\sigma}
\def\Sig{\Sigma}
\def\t{\tau}
\def\th{\theta}

\def\ov{\overline}

\def\ul{\underline}

\title{Continuous Cluster Categories {I}}
\author{Kiyoshi Igusa}
\thanks{The first author is supported by a grant from NSA}

\author{Gordana Todorov}


\begin{document}

\begin{abstract}
In \cite{IT09} we constructed topological triangulated categories $\cC_c$ as stable categories of certain topological Frobenius categories $\cF_c$. In this paper we show that these categories have a cluster structure for certain values of $c$ including $c=\pi$. The \emph{continuous cluster categories} are those $\cC_c$ which have cluster structure. We study the basic structure of these cluster categories and we show that $\cC_c$ is isomorphic to an orbit category $\cD_r/\ul F_s$ of the \emph{continuous derived category} $\cD_r$ if $c=r\pi/s$. In $\cC_\pi$, a \emph{cluster} is equivalent to a discrete lamination of the hyperbolic plane. We give the representation theoretic interpretation of these clusters and laminations.
\end{abstract}

\maketitle
 
 
 








\section*{Introduction} The continuous cluster categories $\cC_c$, by construction, may be considered to be cluster categories of type $A_\RR$. (See Section \ref{ss11}, Theorem \ref{C has cluster structure} and Theorem \ref{cluster structures for Cc}.) These are continuous versions of the cluster category of type $A_n$. Before describing the construction, we recall the standard definitions.

Cluster categories were introduced by Buan, March, Reineke, Reiten and Todorov in \cite{BMRRT}. Given any hereditary algebra $\LL$ over an algebraically closed field $\kk$, the cluster category $\cC(\LL)$ is an orbit category of the derived category $\cD^b(mod\,\LL)$ of bounded complexes over $\LL$ under the triangulated autoequivalence $F=\t^{-1}[1]$. This orbit category inherits the structure of a triangulated category from that of the derived category \cite{KellerOrbit}. Clusters in the cluster category are defined to be maximal collections of nonisomorphic indecomposable objects $T_i$ which are rigid and do not extend each other. Thus $Hom(T_i,\Sig T_j)=0$ for all $i,j$ where $\Sig X=T_j[1]\cong \t X$ in the cluster category.

In special case where $\LL$ is a hereditary algebra of type $A_n$ there was also a combinatorially construction by Caldero, Chapoton and Schiffler in \cite{CCS06}. In that construction, the indecomposable objects of the category are internal chords in a regular $(n+3)$-gon, morphisms are given by counterclockwise rotation and a cluster is a maximal collection of noncrossing internal chords giving a triangulation of the $(n+3)$-gon.

We will first give the idea of continuous cluster categories and a topological interpretation in 0.1, and an outline of the algebraic construction in 0.2, followed by a detailed description of the contents of the paper in 0.3.

\subsection{Idea of continuous cluster category}

The continuous cluster category $\cC=\cC_\pi$ was originally obtained by taking the cluster category of type $A_n$ and taking the limit as $n$ goes to infinity. Although our current construction is different, the idea is still the same. The Auslander-Reiten quiver of the cluster category of type $A_n$ embeds in a Moebius band as a discrete lattice with irreducible morphisms drawn as perpendicular line segments of equal length. As $n$ goes to infinity, the lattice of points (indecomposable objects) converges to all points in the open Moebius band (the points on the boundary are the projective injective objects of the Frobenius category and are therefore zero in the cluster category) and the lengths of the line segments representing irreducible morphisms goes to 0. 

The limiting category is a topological category. Since $\t X\cong \Sig X=X[1]$ converges to $X$ as $n\to \infty$, we have $\Sig X\cong X$ in the limit. However, a continuous triangulation of a topological additive category does not allow $\Sig X=X$ (except when the characteristic of $\kk $ is 2). So, we must have at least two objects in each isomorphism class of indecomposable objects. The minimum topological space must be a two-fold covering space of the open Moebius band. There are two possibilities (as shown in \cite{IT09}). One is the connected oriented covering which is geometrically easier to describe. Namely, we view the indecomposable objects as geodesics in the hyperbolic plane and the two isomorphic copies are represented by the two possible orientations of the corresponding geodesic. 

Clusters in the continuous cluster category are certain maximal sets of compatible objects. In the geometric picture, two geodesics are compatible if they do not cross. A maximal collection of noncrossing geodesics is a ``lamination''. However, not all laminations can be considered to be clusters since clusters need the crucial property that any object in a cluster can be exchanged for exactly one other object (up to isomorphism). Since laminations are closed sets, any object which is a limit of other objects in the lamination cannot be exchanged. So, we define a \emph{cluster} to be a discrete lamination (they are also known as ``simple laminations''). These correspond to ideal triangulations of the hyperbolic plane. Algebraically, two indecomposable objects $X,Y$ are \emph{compatible} if any nonzero morphism $X\to Y$, when completed to a distinguished triangle $X\to Y\to Z\to \Sig X\cong X$, makes the third  object $Z$ indecomposable.
 
\subsection{Construction of continuous cluster category}

In this paper we construct the continuous cluster category $\cC=\cC_\pi$ as an orbit category of the \emph{doubled continuous derived category} $\cD_\pi\double$, defined below, under a triangulated functor $\ul F_\pi$. More generally we take any $0<r\le s$ and define $\cC_{r,s}$ to be the orbit category $\cD_r\double/\ul F_s$, then we show that $\cC_{r,s}$ is isomorphic to the stable category $\cC_c$ of the Frobenius category $\cF_c$ constructed in \cite{IT09} where $c=r\pi/s$.

For any positive real number $r$ we can define $\cD_r$ to be the $\kk $-linear additive Krull-Schmidt category with indecomposable objects $M(x,y)$ where $x,y$ are real numbers with $|x-y|<r$. Hom sets are given by 
\[
\cD_r(M(x,y),M(a,b))=\begin{cases} \kk  & \text{if $x\le a<y+r$ and $y\le b<x+r$}\\
  0  & \text{otherwise}
    \end{cases}
\]
To construct $\cC_{r,s}$, we first ``double'' the category $\cD_r$ by taking two isomorphic copies of each indecomposable object, call them $M(x,y)_+, M(x,y)_-$. This doubled category is denoted $\cD_r\double$. Then $\cC_{r,s}$ is given by identifying $M(x,y)_\e$ with $\ul F_s M(x,y)_\e=M(y+s,x+s)_{-\e}$. This is accomplished by taking the orbit category: $\cC_{r,s}=\cD_r/\ul F_s$. It is easy to see that, up to isomorphism, $\cC_{r,s}$ depends only on the ratio $\frac{r}{s}$. To see the geometric picture, we take $s=\pi$ and view $M(x,y)_+$ as a directed geodesic from $x$ to $y+\pi$ on the circle $S^1=\RR/2\pi\ZZ$ and $M(x,y)_-$ as the same geodesic oriented the other way.

To give $\cD_r$ the structure of (continuously) triangulated categories, we construct a topological Frobenius category $\cB_{\le r}$ whose stable category is $\cD_r$. We also show that $\cF_c$ is isomorphic to the completed orbit category of $\cB_{\le r}$ under a triangulated functor $F_s$ whenever $c={r\pi/s}$.

As an additive category, $\cB_{\le r}$ is easy to describe. The indecomposable objects are $M(x,y)$ where $x,y$ are real numbers with $|x-y|\le r$ and the hom sets are
\[
\cB_{\le r}(M(x,y),M(a,b))=\begin{cases} \kk  & \text{if $x\le a$ and $y\le b$}\\
  0  & \text{otherwise}
    \end{cases}
\]
In other words, all morphism go right and up. Composition is given by multiplication of scalars. In other words, composition of basic morphisms (corresponding to $1\in \kk $) is basic unless it is forced to be zero. To show that this is a Frobenius category we need an exact structure. This is given by embedding the category in the abelian category of representations of the real line. We show in a later paper \cite{IT13} that the exact structure is uniquely determined.

Recall from \cite{HappelBook} that the stable category of a Frobenius category has the same set of objects and Hom sets given by quotienting out all morphisms which factor through a projective-injective object. For $\cB_{\le r}$ the projective injective objects are $M(x,y)$ where $|x-y|=r$.


\subsection{Contents of the paper} We describe the contents of the paper with more technical details.

In Section 1 we construct the continuous derived category $\cD_r$. We start with the abelian category $\cA$ of representations of the real line over the field $\kk $. An indecomposable object $V_S$ of $\cA$ is uniquely determined up to isomorphism by its support $S$ which must be a connected subset of $\RR$. We construct a strictly additive full subcategory $\cA_\RR\sa=add\,\cX$ where $\cX$ is the full subcategory of $\cA$ generated by objects $V_{(a,b]}$ where $a<b$ and we give it an obvious topology. The category $\cB$ is isomorphic to the exact category $\cA_\RR'$ which is the extension closed subcategory of $add\,\cX$ generated by all $V_{(a,b]}$ where $a<0<b$. The isomorphism sends $V_{(a,b]}$ to the object $M(x,y)$ in $\cB$ where $x=-ln(-a)$ and $y=ln\,b$.  Since $\cB\cong \cA_\RR'$ it is exact.

The Frobenius category $\cB_{\le r}$ is defined to be the strictly additive full category of $\cB$ generated by all $M(x,y)$ where $|x-y|\le r$. The stable category of $\cB_{\le r}$ is the continuous derived category $\cD_r$. Being the stable category of a Frobenius category it admits a triangulated structure which is specified by a choice of two-way approximation sequences $X\to IX\to \Sig X$ for every object $X$. This means that $IX$ is a left approximation of $X$ by a projective-injective object and a right approximation of $\Sig X$. We give the formula for our choice of two-way approximations and show that it leads to ``positive triangles'' as the basic distinguished triangles in the category. The idea of using approximations to define triangulated structures comes from \cite{BM94}.

In Section 2 we double the categories $\cD_r$ and $\cB_{\le r}$ and show that the orbit category of $\cD_r\double$ under a strictly triangulated strictly additive functor $\ul F_s:\cD_r\double\to\cD_r\double$ is equivalent to the stable category $\cC_c$ of the continuous Frobenius category $\cF_c$ constructed in \cite{IT09} for $c=r\pi/s$. To do this we take the completed orbit category $\cB_{\le r}\double/F_s^\wedge$ whose morphism sets are products \[
\cB_{\le r}\double/F_s^\wedge(X,Y)=\prod_{n\in\ZZ}\cB_{\le r}\double(X,F_s^nY)\]
The equivalence $\hat G:\cB_{\le r}\double/F_s^\wedge\to \cF_{r\pi/s}$ is induced by a continuous strictly additive functor $G:\cB_{\le r}\double\to \cF_{r\pi/s}$ which is uniquely determined by the restriction of its object map to the indecomposable objects (by an argument similar to the one in Prop \ref{uniqueness of continuous functors}).

In Section 3 we consider maximal compatible sets of indecomposable objects of the continuous cluster category $\cC=\cC_\pi=\cC_{\pi,\pi}$. We show that these are laminations of the hyperbolic plane. We define a \emph{cluster} to be a discrete lamination. In Section 4 we show that all clusters are equivalent under strictly additive automorphisms of the cluster category. We also describe all such automorphisms. In section 3 we also give a short proof of a theorem of W. Thurston \cite{Th} showing that laminations are locally homeomorphic to compact subsets of $\RR$ which are arbitrary.

Finally, in Section 5, we show that clusters in $\cC$ form a cluster structure as defined in \cite{BIRSc}. We also show that mutation of clusters is given by the octahedral axiom. For $r<s$, the continuous orbit category $\cC_{r,s}$ also has a cluster structure if and only if $\frac{r}{s}=(n+1)/(n+3)$.


%
%

%
%


\section{The continuous derived category}


We will construct an exact category $\cB$ as an extension closed subcategory of an abelian category $\cA_\RR$ and choose certain full subcategory $\cB_{\le r}$ for $c>0$ which we show to be Frobenius categories with $\cB_r$ being the full subcategory of projective-injective objects. Then the stable category $\cD_r=\ul{\cB}_{\le r}$ is a triangulated category which we call the \emph{continuous derived category}. We briefly examine the topological structure of these category. We also examine some of the distinguished triangles in $\cD_r$. All of these categories will be $\kk $-categories where $\kk $ is a fixed field. We assume there is a unique zero object in the category of $\kk $-vector spaces. (Choose one $0$ and discard all the others.)

\subsection{The abelian category $\cA$}\label{ss11}

The category $\cA_\RR$ will be defined to be the category of representations of the ``quiver'' $\RR$. This is similar to the representations of infinite linearly ordered posets considered in \cite{vR}, except that $\RR$ is not locally finite, so we do not have Serre duality. 

\begin{defn}
A \emph{representation} of $\RR$ over a field $\kk $ is defined to be a functor $V$ which assigns a $\kk $-vector space $V_x$ to every $x\in\RR$ and a $\kk $-linear map $V_{xy}:V_y\to V_x$ for all $x< y$ so that
\begin{enumerate}
\item $V_{xy}V_{yz}=V_{xz}$ for all $x< y<z$ and
\item $(\forall x\in\RR)(\exists z<x)(\forall y\in(z,x))V_{yx}:V_x\to V_y$ is an isomorphism.
\end{enumerate}
A morphism between two such representations $f:V\to W$ is a collection of linear maps $f_x:V_x\to W_x$ so that $W_{xy}f_y=f_xV_{xy}$ for all $x<y$.
\end{defn}

Given any representation $V$, we define a \emph{regular value} of $V$ to be any real number $x$ so that $V_{xy}:V_y\to X_x$ is an isomorphism for all $y>x$ sufficiently close to $x$. We define the \emph{critical values} of $V$ to be the other elements of $\RR$. Let $\cA_{\RR}$ be the category of representations $V$ of $\RR$ over $\kk$ which are finite dimensional at every point of $\RR$ and which have only finitely many critical values. For every $S$ of $\RR$ of the form $(a,b], (a,\infty),(-\infty,b]$ or $(-\infty,\infty)=\RR$ we define the representation $V_S$ of $\RR$ as $V_x=\kk$ for all $x\in S$ and $V_{xy}:V_y=\kk\to V_x=\kk$ is the identity map for all $x,y\in S$. Then $V_S$ is an object of $\cA_\RR$.

The following observation can be used to reduce statements involving finitely many objects of the category $\cA_\RR$ to statements about representations of the quiver $A_{n}$ with straight orientation:
$
	1\ot 2\ot \cdots\ot n
$.
\begin{prop}
For any finite subset $a_\ast=\{a_1,a_2,\cdots,a_n\}$ of $\RR$, let $\cA_{a_\ast}$ be the full subcategory of $\cA_\RR$ of all representations whose critical values lie in the set $a_\ast$. Then $\cA_{a_\ast}$ is equivalent to the category of representations of the quiver $A_{n+1}$ with straight orientation as above. The quiver representation $W$ corresponding to $V\in\cA_{a_\ast}$ is given by $W_i=V_{a_i}$ for $ i\le n$ and $W_{n+1}=V_{a_n+1}$.\qed
\end{prop}

Given the well-known properties of the quivers of type $A_n$ we get the following.

\begin{cor}\begin{enumerate}
\item $\cA_\RR$ is an abelian Krull-Schmidt category. 
\item The projective objects of $\cA_\RR$ are those in which the structure maps $V_{xy}$ are monomorphism for all $x<y$. Indecomposable projective objects are isomorphic to $P_a:=V_{(-\infty,a]}$ or $V_\RR$.
\item $\cA_\RR$ is hereditary (all subobjects of projective objects are projective).
\item The injective objects are those for which $V_{xy}$ is an epimorphism for all $x<y$. Indecomposable injective objects are isomorphic to $I_a:=V_{(a,\infty)}$ or $V_\RR$.
\item The remaining indecomposable objects are isomorphic to $V_{(a,b]}$ for some $a<b$.\qed
\end{enumerate}
\end{cor}

Continuing our list of easy properties of the category $\cA_\RR$ we have:

\begin{cor}\label{inherited properties of representations of type An}
\begin{enumerate}
\item $\cA_\RR$ has enough projectives and injectives.
\item Homomorphisms between the standard indecomposable objects $V_{(a,b]}$ are given by scalars in $\kk$:
\[
	\cA(V_{(a,b]},V_{(c,d]})=\begin{cases}
	\kk  & \text{if } a\le c<b\le d\\
	0 & \text{otherwise}
	\end{cases}
\]
\item Composition of morphisms between standard indecomposable objects is given by multiplication of scalars.
\item
\[
	\Ext_\cA(V_{(c,d]},V_{(a,b]})\cong\begin{cases}
	\kk  & \text{if } a< c\le b< d\\
	0 & \text{otherwise}
	\end{cases}
\]
\end{enumerate}
\end{cor}

\begin{proof}
This follows from very well-known facts about representations of quivers of type $A_n$. We should only point out that the homomorphism $V_{(a,b]}\to V_{(c,d]}$ corresponding to the scalar $s\in \kk $ is given by the composition
\[
	V_{(a,b]}\onto V_{(c,b]}\xrarrow{s}V_{(c,b]} \into V_{(c,d]}
\]
This is $s$ times the \emph{basic morphism} which is the identity on $V_{(c,b]}$ and zero outside this interval and statement (3) follows from the fact that any composition of basic morphisms is a basic morphism or zero. Also the nonzero extensions of $V_{(a,b]}$ by $V_{(c,d]}$ all have the form:
\[
	V_{(a,b]}\cof V_{(c,b]}\oplus V_{(a,d]}\onto V_{(c,d]}\to 0
\]
where $V_{(c,b]}:=0$ when $c=b$ and we will discuss later the choices of the morphisms in the basic extension sequence.
\end{proof}

\subsection{Strictly additive and topological categories}

The category $\cA_\RR$ is equivalent to a ``strictly additive'' full subcategory $\cA_\RR\sa$ defined to be the additive category generated by the standard indecomposable objects $V_{(a,b]}$. Thus objects of $\cA_\RR\sa$ are finite ordered sequences of standard objects $V_{(a_1,b_1]}\oplus\cdots \oplus V_{(a_n,b_n]}$ where $n\ge0$.

An additive category $\cC$ is defined to be \emph{strictly additive} if direct sum is strictly associative on objects and morphisms and has a strict identity. (So, $\cC$ has a distinuished zero object $0$ so that $0\oplus X=X=X\oplus 0$.) A functor $\Phi$ between strictly additive categories will be called \emph{strictly additive} if it strictly commutes with direct sum and takes the distinguished zero object of the first category to the distinguished zero object of the second category. So, $\Phi(X\oplus Y)= \Phi  X\oplus \Phi Y$ and $\Phi(f\oplus g)= \Phi f\oplus \Phi g$. A strictly additive functor on a strictly additive Krull-Schmidt category is uniquely determined by its value on the full subcategory of indecomposable objects since $\Phi$ maps any morphism $\sum f_{ij}:\oplus X_j\to \oplus Y_i$ to $\sum \Phi f_{ij}:\oplus \Phi X_j\to \oplus\Phi Y_i$ because 
\[
f_{ij}=0\oplus f_{ij}\oplus 0:(\cdots)\oplus X_j\oplus(\cdots)\to(\cdots)\oplus Y_i\oplus(\cdots)
\]
which maps to $0\oplus \Phi f_{ij} \oplus 0$ since $\Phi$ is strictly additive.

This emphasis on strictly additive categories is motivated by topological considerations. Recall that a \emph{topological category} is a small category together with a topology on the set of objects and the set of all morphisms so that the structure maps of the category (source, target, composition) are continuous mappings. A \emph{continuous functor} between topological categories is a functor $\Phi$ which is a continuous mapping on object sets and morphism sets. A \emph{topological $\kk$-category} is a topological category which is also a $\kk$-category where addition and scalar multiplication of morphisms is continuous.

If $\cC$ is a topological category then $add\,\cC$ is also a topological category with object set $\coprod_{n\ge0} Ob(\cC)^n$ with product topology on $Ob(\cC)^n$ and morphism set the subspace of \[\coprod_{n,m\ge0}Ob(\cC)^n\times Mor(\cC)^{nm}\times Ob(\cC)^m\]consisting of triples $((X_i)_{1\le i\le n},(f_{ij})_{1\le i\le n,1\le j\le m},(Y_j)_{1\le j\le m})$ so that each $f_{ij}\in \cC(X_i,Y_j)$. Any continuous functor between topological additive categories $\Phi:\cC\to \cD$ induces a strictly additive continuous functor $add\,\cC\to add\,\cD$ also called $\Phi$ by $\Phi((X_i),(f_{ij}),(Y_j))=((\Phi X_i),(\Phi f_{ij}),(\Phi Y_j))$.

As an example, let $\cX$ be the $\kk$-category of standard objects $V_{(a,b]}$. Then
\[
Ob(\cX)= \{(a,b)\in \RR^2\,|\, a<b\}.\]
 We take the usual metric topology on this space. The morphism set of $\cX$ is $Mor(\cX)=\kk \times M$ where \[
 M= \{(a,b,c,d)\in \RR^4\,|\, a<b, c<d, a\le c, b\le d\}\subset Ob(\cX)^2\]
 We give $\kk$ the discrete topology and $\RR^4$ the usual topology. This defines a topological $\kk$-category structure on $\cX$ with the key property that $Mor(\cX)$ is a covering space of $M\subseteq Ob(\cX)^2$. Thus $\cA_\RR\sa=add\,\cX$ is a topological additive category.
 
 As an example of the use of this structure we have the following.
 
 \begin{prop}\label{uniqueness of continuous functors}
 Any continuous strictly additive $\kk$-linear endofunctor on $\cA_\RR\sa$ is uniquely determined by its value on indecomposable objects.
 \end{prop}
 
 \begin{proof}
 If an endofunctor $\Phi$ sends $X$ to $\Phi X$, it must send the identity morphism of $X$ to the identity morphism of $\Phi X$. Since $\cX(X,X)\cong \kk $ and $\Phi$ is $\kk$-linear, the value of $\Phi$ on every endomorphism of $X$ is determined. This determines the value of $\Phi$ on one point in each sheet of the covering space $Mor(\cX)\to M\subset Ob(\cX)$. Since $M$ is path connected, there is a unique continuous map $Mor(\cX)\to Map(\cX)$ which covers the given object map $\Phi:M\to M$. So, $\Phi$, if continuous, is uniquely determined on $Mor(\cX)$. Since $\Phi$ is strictly additive, its value on all of $add\,\cX=\cA_\RR\sa$ is uniquely determined.
 \end{proof}


\subsection{The exact category $\cB$}

We will construct a strictly additive category $\cB$ which is isomorphic to a full subcategory $\cA_\RR'$ of $\cA_\RR^{sa}$. Since $\cA_\RR'$ is an extension closed full subcategory of the abelian category $\cA_\RR^{sa}$ our category $\cB$ will be an exact category.

\begin{prop}\label{characterization of objects of B}
Suppose that $V$ is an object of $\cA_\RR$ with the following properties.
\begin{enumerate}
\item $0$ is a regular value of $V$.
\item $V_x=0$ for $|x|$ sufficiently large.
\item $V_{0x}$ is a monomorphism for all $x>0$.
\item $V_{x0}$ is an epimorphism for all $x<0$.
\end{enumerate}
Then $V$ is isomorphic to a direct sum of representations of the form $V_{(a,b]}$ where $a<0<b$ (with $a,b$ both finite). Conversely, all such representations have the properties listed above.\qed
\end{prop}

We define $\cA_\RR'$ to be the additive full subcategory of $\cA_\RR\sa$ generated by the indecomposable nonprojective and noninjective objects $V_{(a,b]}$ where $a<0<b$. In other words, $\cA_\RR'$ is the strictly additive version of the full subcategory of $\cA_\RR$ described by the above proposition. 

\begin{prop}
$\cA_\RR'$ is an extension closed subcategory of $\cA_\RR^{sa}$.
\end{prop}

\begin{proof}
Suppose that $A\cof B\onto C$ is a short exact sequence in $\cA_\RR^{\sa}$ and $A,C$ lie in $\cA_\RR'$. Consider Condition (3). For any $x>0$ we have the following commuting diagram with exact rows.
\[
\xymatrix{
0\ar[r] &
	A_0\ar[d]^{A_{0x}}\ar[r] &B_0\ar[d]^{B_{0x}}\ar[r] &C_0\ar[d]^{C_{0x}}\ar[r] & 0\\
0 \ar[r]& 
	A_x \ar[r]&B_x \ar[r]&C_x \ar[r]&
	0
	}
\]
Since $A_{0x},C_{0x}$ are monomorphisms, so is $B_{0x}$. So, $B$ satisfies condition (3). The other conditions are verified in a similar way and we conclude that $B$ is an object of $\cA_\RR'$.
\end{proof}

Thus, $\cA_\RR'$ is an exact category. The exact sequences of $\cA_\RR'$ are defined to be the short exact sequences in $\cA_\RR\sa$ all of whose objects lie in $\cA_\RR'$.

There is a better description of the exact sequences in the exact category $\cA_\RR'$. Roughly speaking it says that a sequence of morphisms $A\cof B\onto C$ is a short exact sequence in $\cA_\RR'$ if and only if it is split exact on the positive real numbers and on the negative real numbers. To state this more precisely, let $\pi_+,\pi_-:\cA_\RR'\to \cA_\RR$ be the exact functors given by $\pi_+V_{(a,b]}=P_y$ where $y=log\,b$ and $\pi_-V{(a,b]}=P_x$ where $x=-log|a|$ and which take morphisms to the corresponding morphisms. Then the image of each of these functors is the full subcategory $\cP_\RR$ of all projective objects in $\cA_\RR^{sa}$ with no injective summands (i.e., which are zero at large positive real numbers). Then $\cA_\RR'$ is isomorphic to the pull-back in the diagram:
\[
\xymatrix{
\cA_\RR'\ar[r]^{\pi_+}\ar[d]_{\pi_-}& \cP_\RR\ar[d]^\e\\
\cP_\RR\ar[r]^(.4)\e & \kk \- mod
}
\]
where $\e:\cP_\RR\to \kk \- mod$ is the exact functor which sends each $P_x$ to $\kk$. 

Since all exact sequences in $\cP_\RR$ split, we have the following.

\begin{prop}\label{characterization of exact sequences in B}
A sequence of morphisms $X\cof Y\onto Z$ in $\cA_\RR'$ is a short exact sequence if an only if it is split exact in each coordinate, i.e., if and only if $\pi_+(X)\cof\pi_+(Y)\onto \pi_+(Z)$ and $\pi_-(X)\cof\pi_-(Y)\onto \pi_-(Z)$ are split exact sequences in $\cP_\RR$.\qed
\end{prop}

We now define the category $\cB$ which will be isomorphic to $\cA_\RR'$.

\begin{defn} Let $\cB$ be the strictly additive $\kk$-category with one indecomposable object $M(x,y)$ for every ordered pair of real numbers $(x,y)$. Morphism sets are defined by
\[
	\Hom(M(x,y),M(x',y'))=\begin{cases}
	\kk  & \text{if } x\le x'\text{ and }y\le y'\\
	0 & \text{otherwise}
	\end{cases}
\]
The morphism $M(x,y)\to M(x',y')$ corresponding to $1\in\kk$ will be called the \emph{basic morphism}. We define any composition of basic morphisms to be a basic morphism. Thus morphisms between indecomposable objects are scalar multiples of basic morphisms and composition is given by multiplication of these scalars.
\end{defn}

\begin{prop}
There is a $\kk$-linear isomorphism of categories $\cB\to \cA_\RR'$ given by sending $M(x,y)$ to $V_{(a,b]}$ where $a=-e^{-x},b=e^y$ and basic morphisms to basic morphisms. Therefore, $\cB$ is an exact category.\qed
\end{prop}

We note that the compositions $\cB\cong\cA_\RR'\xrarrow{\pi_-}\cP_\RR$ and $\cB\cong\cA_\RR'\xrarrow{\pi_+}\cP_\RR$ send $M(x,y)$ to $P_x$ and $P_y$ respectively. By Proposition \ref{characterization of exact sequences in B}, a sequence in $\cB$ is exact if and only it its images under there two functors are both exact.

From Corollary \ref{inherited properties of representations of type An}, we also have
\[
	\Ext(M(x',y'),M(x,y))\cong \begin{cases}
	\kk  & \text{if } x< x'\text{ and }y< y'\\
	0 & \text{otherwise}
	\end{cases}
\]
with all nonsplit extensions being isomorphic to the exact sequence
\[
\xymatrixrowsep{10pt}\xymatrixcolsep{25pt}
\xymatrix{
M(x,y)\ \ar@{>->}[r]^(.35){\binom t{-t}} &
	M(x',y)\oplus M(x,y')\ar@{->>}[r]^(.6){(1,1)} &
	M(x',y')
	}
\]
for some $t\in \kk ^\times$.

From this description it is clear that $\cB$ is completely \emph{homogeneous} by which we mean that the group of automorphisms of $\cB$ acts transitively on the set of objects. In fact, if $\f_1,\f_2$ are two order preserving bijections (= orientation preserving homeomorphisms) of $\RR$ then we get a strictly additive $\kk$-linear automorphism of $\cB$ given on indecomposable objects by
\[
	M(x,y)\mapsto M(\f_1(x),\f_2(y))
\]
and on morphisms by the property that it takes basic morphisms to basic morphisms. Thus $\Aut(\cB)$ acts transitively on the set of objects of $\cB$. This is an example of the following definition.

\begin{defn}
Suppose that $\f:\RR^2\to\RR^2$ is a function which is nondecreasing in both coordinates in the sense that if $x\le x'$ and $y\le y'$ then $\f_1(x,y)\le\f_1(x',y')$ and $\f_2(x,y)\le \f_2(x',y')$ where $\f(x,y)=(\f_1(x,y),\f_2(x,y))$. Then we define $\Phi_\f$ to be the strictly additive endomorphism of $\cB$ (always assumed to be $\kk$-linear) given on indecomposable objects by $\Phi_\f M(x,y)=M(\f(x,y))$ and given uniquely on morphisms by the condition that it takes basic morphisms to basic morphisms. We call $\Phi_\f$ the \emph{basic endomorphism} induced by the object map $\f$.
\end{defn}

As another example the \emph{involution} $\s$ on $\cB$ is the basic automorphism given on objects by switching $x$ and $y$: $\s M(x,y)=M(y,x)$. This is equivalent to saying that $\s$ reverses the projection functors on the first and second coordinates: $\s\pi_+=\pi_-$ and $\s\pi_-=\pi_+$.

A more general endomorphism of $\cB$ corresponding to the object map $\f$ is given by twisting $\Phi_\f$ by a coefficient map $\ll:\RR^2\to \kk ^\times$. We define the \emph{twisted endomorphism} $\Phi_{(\f,\ll)}$ to be the strictly additive endomorphism of $\cB$ which is given by $\f$ on objects and on morphisms is given by sending the basic morphism $M(x,y)\to M(x',y')$ to $\ll(x',y')/\ll(x,y)$ times the basic morphism $M(\f(x,y))\to M(\f(x',y'))$. Since we only use the ratio of two values of $\ll$, we may change $\ll$ by multiplying all values by the same scalar in $\kk ^\times$ without changing the corresponding $\Phi_{(\f,\ll)}$. Note that $\Phi_{(\f,\ll)}$ is continuous if and only if $\ll$ is constant (so that $\Phi_{(\f,\ll)}=\Phi_\f$ is untwisted).

\begin{prop} All strictly additive automorphism of $\cB$ are given by $\Phi_{(\f,\ll)}$ where $\f$ is given by either $\f(x,y)=(\f_1(x),\f_2(y))$ for all $x,y$ or $\f(x,y)=(\f_2(y),\f_1(x))$ for all $x,y$ with $\f_1,\f_2$ being any order preserving bijection of $\RR$ and $\ll:\RR^2\to \kk ^\times$ any mapping.
\end{prop}

\begin{proof}
Suppose that $\Phi$ is a strictly additive automorphism of $\cB$. Then we will show that $\Phi$ is one of the automorphisms that we have described.

Composing with a suitable object map $\f=(\f_1,\f_2)$ we may assume that $\Phi M(0,0)=M(0,0)$. Using $\Hom$ and $\Ext$ we see that $\Phi$ preserves the set of objects $M(x,y)$ for $x,y\ge 0$ and the set of all $M(x,y)$ where $x,y>0$. Therefore, $\Phi$ preserves the union of the positive $x$-axis and the positive $y$-axis. Since the objects $M(x,0),M(0,y)$ for $x,y>0$ do not map nontrivially to each other, we see that $\Phi$ either preserves both of these sets (positive $x$-axis and positive $y$-axis) or switches them. Composing with $\s$ is necessary, we may assume that $\Phi$ preserves the set of all $M(x,0)$ where $x>0$ and the set of all $M(0,y)$ where $y>0$. By a similar argument we may also assume at the same time that $\Phi$ preserves the negative $x$-axis and the negative $y$-axis. Therefore, $\Phi M(x,0)=M(\f_1(x),0)$ and $\Phi M(0,y)=M(0,\f_2(y))$ for order preserving bijections $\f_1,\f_2$ of $\RR$. Then it follows that $\Phi M(x,y)=M(\f_1(x),\f_2(y)$ for all $x,y$ since this is the only indecomposable module which has homomorphisms but no extensions from $\Phi M(x,0)$ and $\Phi M(0,y)$. Thus we may assume that $\Phi$ is the identity map on all objects of $\cB$.

If $\Phi$ is the identity map on objects then $\Phi$ is given by linear automorphisms of \[
	\cB(M(x,y),M(x',y'))=\kk 
\]for all $x\le x'$ and $y\le y'$. These are given by scalar multiplication by some $f((x,y),(x',y'))\in \kk $ which must satisfy the following two conditions in order to form a functor:
\[
	f((x,y),(x',y'))f((x',y'),(x'',y''))=f((x,y),(x'',y''))
\] 
and $f(x,y),(x,y))=1$. Let $\ll:\RR^2\to \kk ^\times$ be given by
\[
	\ll(x,y)=\frac{f((0,0),(|x|,|y|))}{f((x,y),(|x|,|y|))}
\]
Then one can easily show that
\[
	f((x,y),(x',y'))=\ll(x',y')/\ll(x,y)
\]
proving the proposition.
\end{proof}



\subsection{The category $\cB_{\le r}$}

For any real number $r>0$ we define $\cB_{\le r}$ to be the additive full subcategory of $\cB$ generated by all indecomposable objects of the form $M(x,y)$ where $|y-x|\le r$. All such choices are equivalent in the sense that there is a strictly additive automorphism of $\cB$ which sends $\cB_{\le r}$ to $\cB_{\le r'}$ for any $r, r'>0$, namely the functor $\Phi_{\f}$ where $\f(x,y)=(r'x/r,r'y/r)$. Let $\cB_r$ be the full subcategory of $\cB_{\le r}$ generated by all $M(x,y)$ where $|y-x|=r$. 

We will show that $\cB_{\le r}$ is a Frobenius category by showing that it is an exact category in which $\cB_r$ is the full subcategory of projective-injective objects. We also choose a specific two-way $\cB_r$ approximation sequence for each object of $\cB_{\le r}$ in order to specify the triangulated structure of its stable category of $\ul\cB_{\le r}$. 

For any $x,y\in\RR$ with $x-r<y<x+r$ consider the following exact sequence which we call the \emph{positive two-way approximation sequence}. (The \emph{negative two-way approximation sequence} is given by replacing $\binom{1}{-1}$ with $\binom{-1}{1}$.)
\begin{equation}\label{positive two-way approximation}
\xymatrixrowsep{10pt}\xymatrixcolsep{25pt}
\xymatrix{
M(x,y)\ \ar@{>->}[r]^(.3){\binom{1}{-1}} &
	M(y+r,y)\oplus M(x,x+r)\ar@{->>}[r]^(.6){(1,1)} &
	M(y+r,x+r)
	}
\end{equation}
Note that only the terms in the middle lie in $\cB_r$. These points are plotted in Figure \ref{positive two-way approximation}.
\begin{figure}[htbp]
\begin{center}
%
{
\setlength{\unitlength}{1cm}
{\mbox{
\begin{picture}(5,3)
    \thinlines
    \put(0,-.5){
      \qbezier(3.3,.8)(4,1.5)(4.5,2)
    \put(3.5,1){$\bullet$}
      \put(3.8,1){$M(y+r,y)$}
      \put(5,1.7){$\cB_r$}
      \put(-1,1.8){$\cB_r$}
\put(0,-.5){
      \qbezier(0.2,2.2)(1,3)(1.5,3.5)
    \put(3.5,3){$\bullet$}
      }
      }
        \thicklines
        \put(0,-.5){
    \put(1,1){$\bullet$}
      \put(-.4,1){$M(x,y)$}
      \put(0,-.5){
    \put(1,3){$\bullet$}
      \put(-1.1,3){$M(x,x+r)$}
      \put(3.8,3){$M(y+r,x+r)$}}
             \put(2.3,2.6){\qbezier(-1.2,0)(0,0)(1.2,0) 
             \put(-.2,0){\qbezier(.7,0)(.5,0)(.4,0.13) 
             \qbezier(.7,0)(.5,0)(.4,-0.13)             }
             }
             \put(1.1,1.1){\qbezier(0,0)(1,0)(2.4,0)
             \put(1,0){\qbezier(.7,0)(.5,0)(.4,0.13) 
             \qbezier(.7,0)(.5,0)(.4,-0.13)}
             }
             \put(1.1,.6){\qbezier(0,0.5)(0,1)(0,2) 
             \qbezier(0,1.7)(0,1.5)(-.13,1.4) 
             \put(-.6,1){$-1$}
             \qbezier(0,1.7)(0,1.5)(.13,1.4)}
             \put(3.6,.6){\qbezier(0,0.5)(0,1)(0,2) 
             \qbezier(0,1.7)(0,1.5)(-.13,1.4) 
             \qbezier(0,1.7)(0,1.5)(.13,1.4)}
             }
\end{picture}}
}}
\caption{Basic positive two-way approximation sequence in $\cB_{\le r}$ (Equation \ref{positive two-way approximation}).}
\label{fig two-way approximation in B}
\end{center}
\end{figure}
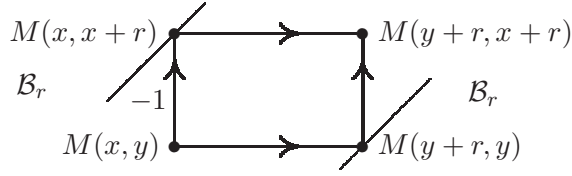

\begin{lem}
The sequence \eqref{positive two-way approximation} is a two-way approximation sequence. I.e., the middle term is a minimal left $\cB_r$ approximation of $M(x,y)$ and a minimal right $\cB_r$ approximation of $M(y+r,x+r)$.
\end{lem}

To show that $\cB_{\le r}$ is an exact category, we use the following general fact.

\begin{lem}\label{lemma for subcategories of exact categories to be exact}
Suppose that $\cD$ is an additive full subcategory of an exact category $\cE$ with the property that, for any exact sequence $X\cof Y\onto Z$ in $\cE$ with middle term $Y$ in $\cD$, the other two terms $X,Z$ also lie in $\cD$. Then $\cD$ is an exact category with exact sequences being those sequences which are exact in $\cE$ with all three terms in $\cD$.
\end{lem}

\begin{proof} Recall that an exact category is an additive category with a distinguished collection of exact sequences $\xymatrixrowsep{10pt}\xymatrixcolsep{12pt}
\xymatrix{
X\ \ar@{>->}[r]^f &
	Y\ar@{->>}[r]^g &
	Z
	}
$ where $X$ is the kernel of $g$ and $Z$ is the cokernel of $f$ and which satisfying the following list of axioms given by Keller \cite{K}.

(E0) $0\cof 0\onto 0$ is an exact sequence in $\cD$ (since $0\in \cD$.)

(E1) Given exact sequences $\xymatrixrowsep{10pt}\xymatrixcolsep{12pt}
\xymatrix{
X\ \ar@{>->}[r]^f &
	Y\ar@{->>}[r] &
	Z
	}
$ and $\xymatrixrowsep{10pt}\xymatrixcolsep{12pt}
\xymatrix{
Y\ \ar@{>->}[r]^g &
	W\ar@{->>}[r] &
	Z'
	}
$ in $\cD$ there is an exact sequence $\xymatrixrowsep{10pt}\xymatrixcolsep{12pt}
\xymatrix{
X\ \ar@{>->}[r]^{g\circ f} &
	W\ar@{->>}[r] &
	Z''
	}
$. (This is true in $\cE$ and $W\in\cD$.)

(E2) The pushout of an exact sequence $\xymatrixrowsep{10pt}\xymatrixcolsep{12pt}
\xymatrix{
X\ \ar@{>->}[r]^f &
	Y\ar@{->>}[r] &
	Z
	}
$ along any morphism $h:X\to X'$ in $\cD$ exists and gives an exact sequence $X'\cof Y'\onto Z$ in $\cD$. (The pushout exists in $\cE$ and is exact. Since $X\cof X'\oplus Y\onto Y'$ is exact in $\cE$ and $X'\oplus Y\in\cD$, the object $Y'$ lies in $\cD$. So, the pushout sequence lies in $\cD$.) 

Similarly, we have the dual axiom:

(E2)$^{op}$ The pullback of an exact sequence in $\cD$ along any morphism in $\cD$ is an exact sequence in $\cD$.

Therefore, $\cD$ is an exact category.
\end{proof}

We also need the following easy lemma.

\begin{lem}
Suppose that $X\cof Y\onto Z$ is an exact sequence in $\cB$. Then, for any component $M{(x,y)}$ of $Z$, there are, not necessarily distinct, components $M{(x,b)}$ and $M{(a,y)}$ of $Y$ which map nontrivially to $M{(x,y)}$.
\end{lem}

\begin{proof} This follows immediately from Proposition \ref{characterization of exact sequences in B} which said that a sequence in $\cB$ is exact if and only if it is split exact in each coordinate.\end{proof}

\begin{lem}
$\cB_{\le r}$ is an exact category.
\end{lem}

\begin{proof}
Suppose that $A\cof B\onto C$ is an exact sequence in $\cB$ so that $B$ lies in $\cB_{\le r}$. Let $M(x,y)$ be a component of $C$. Then we will show that $|y-x|\le r$.

By the previous lemma, $B$ has components $M(x,b)$ and $M(a,y)$ which map to $M(x,y)$. But this implies that $b\le y$ and $a\le x$. Since $B$ lies in $\cB_{\le r}$ we must have $|x-b|\le r$ and $|y-a|\le r$. So,
\[
	x-r\le b\le y\le a+r\le x+r
\]
and we conclude that $C$ lies in $\cB_{\le r}$. A similar argument shows that $A$ also lies in $\cB_{\le r}$. Lemma \ref{lemma for subcategories of exact categories to be exact} implies that $\cB_{\le r}$ is an exact category.
\end{proof}

\begin{thm}
$\cB_{\le r}$ is a Frobenius category and $\cB_r$ is the full subcategory of projective-injective objects.
\end{thm}

\begin{proof}
It remains only to show that the objects of $\cB_r$ are the projective and injective objects of $\cB_{\le r}$. So, suppose that $A\cof B\onto C$ is an exact sequence in $\cB_{\le r}$ and $C\in \cB_r$. Let $M(x,x+r)$ be a component of $C$. As in the proof of the lemma above, $B$ has a component $M(a,x+r)$ which maps to $M(x,x+r)$ and this is possible only if $a=x$. The case of $M(x,x-r)$ being similar, we see that $B\onto C$ is a split epimorphism. Therefore objects in $\cB_r$ are projective. Conversely, any projective object $P$ is part of a two-way approximation sequence $A\cof B\onto P$ where $B$ lies in $\cB_r$. Since $P$ is projective, the sequence splits making $P$ an object of $\cB_r$. So $\cB_r$ is the full subcategory of projective objects in $\cB_{\le r}$. Similarly, $\cB_r$ is the full subcategory of injective objects in $\cB_{\le r}$. The existence of two-sided approximations means that $\cB_{\le r}$ has enough projectives. Since the projective objects are also injective, $\cB_{\le r}$ is Frobenius. 
\end{proof}




\subsection{The category $\cD_r$} For any $r>0$ we define $\cD_r$ to be the stable category of the Frobenius category $\cF_{\le r}$. Thus $\cD_r$ has the same objects as $\cB_{\le r}$ but $\cD_r(X,Y)=\cB_{\le r}(X,Y)/\sim$ where we quotient out the morphisms $X\to Y$ which factor through an object in $\cB_c$. The equivalence class of a morphism $f:X\to Y$ is denoted $\ul f$.

\begin{prop}
$\cD_r:=\ul{\cB}_{\le r}$ is a triangulated category for any $r>0$. The automorphism $\Sig =[1]$ of $\cD_r$ is induced by the basic automorphism of $\cB_{\le r}$ given by $\Sig M(x,y)=M(y+r,x+r)$.\qed
\end{prop}

We now use Happel's procedure \cite{HappelBook} to construct all distinguished triangles in the triangulated category $\cD_r$. We begin with the explicit choice of two-way $\cB_r$-approximation sequence for any object in $\cD_r$ given on indecomposable objects by (\ref {positive two-way approximation}). Let $\cB_{<r}\subset \cB_{\le r}$ be the strictly additive full subcategory of $\cB$ generated by all $M(x,y)$ where $|y-x|<r$. For any indecomposable object $M(x,y)$ in $\cB_{<r}$ let $I_1M(x,y)=M(y+r,y),I_2M(x,y)=M(x,x+r)$ and $\Sig M(x,y)=M(y+r,x+r)$. Extend $I_1,I_2,\Sig $ additively to all objects of $\cB_{<r}$. Then, for any object $X$ in $\cB_{<r}$, the two-way $\cB_r$-approximation sequence for $X$ in $\cB_{\le r}$ is:
\[
\xymatrixrowsep{10pt}\xymatrixcolsep{25pt}
\xymatrix{
X\ \ar@{>->}[r]^(.3){\binom{1}{-1}} &
	I_1X\oplus I_2X\ar@{->>}[r]^(.6){(1,1)} &
	\Sig X
	}
\]

Every object $X$ in $\cD_r$ lifts uniquely to an object in $\cB_{<r}$ which we also call $X$.
Then any morphism $\ul f:X\to Y$ in $\cD_r$ comes from a morphism $f:X\to Y$ in $\cB_{\le r}$ where $X,Y$ lie in $\cB_{<r}$. The standard diagram for $f$ using the positive two-way approximation sequence for each component of $X$ is the pushout diagram
\[\xymatrix{
        X\ \ar@{>->}[r]^(.35){\binom1{-1}}\ar[d]^f &I_1X\oplus I_2X\ar@{->>}[r]^(.6){(1,1)}\ar[d]^{
        \left(
       u_1,u_2
        \right)
        } &\Sig X\ar[d]^=\\
        Y\ \ar@{>->}[r]^{g}&
        Z\ar@{->>}[r]^{h}&\Sig X
        }
\]
By Happel's construction \cite{HappelBook}, the resulting distinguished triangle in $\cD_r$ is
\[
	X\xrarrow{\ul f} Y\xrarrow{\ul g}Z\xrarrow{\ul h}\Sig X
\]
Recall that all distinguished triangles in $\cD$ are isomorphic to triangles of this form. 





As a basic example, let $X=M(x,y)$ and $Y=M(x,z)$ where $x-r<y<z<x+r$. Then we get the diagram
\[\xymatrix{
        M{(x,y)}\ \ar@{>->}[r]^(.3){\binom1{-1}}\ar[d]^1 &M{(y+r,y)}\oplus M{(x,x+r)}\ar@{->>}[r]^(.6){(1,1)}\ar[d]^{
        \binom{1\ 0}{0\ 1}
        } &M{(y+r,x+r)}\ar[d]^=\\
        M{(x,z)}\ \ar@{>->}[r]^(.3){\binom1{-1}}&
        M{(y+r,z)}\oplus M{(x,x+r)}\ar@{->>}[r]^(.6){(1,1)}&\Sig M{(x,y)}
        }
\]
Since $M(x,x+r)\in\cB_c$ is zero in $\cD_r$, this gives the distinguished triangle
\begin{equation}\label{eq: basic up-down-up triangle}
        M{(x,y)}\xrarrow{1}M{(x,z)}\xrarrow{1}M{(y+r,z)}\xrarrow{1}\Sig M{(x,y)}
\end{equation}
where the 1's represents the fact that each morphism in this triangle is a basic morphism. These points are plotted in Figure \ref{fig updown}. We call this the \emph{basic positive triangle}.
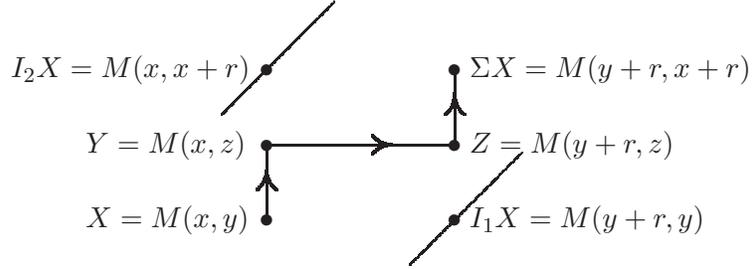
\begin{figure}[htbp]
\begin{center}
%
{
\setlength{\unitlength}{1cm}
{\mbox{
\begin{picture}(5,3.5)
    \thinlines
    \put(0,-.5){
      \qbezier(3,.5)(3.5,1)(4.5,2)
      \qbezier(0.5,2.5)(1,3)(2,4)
    \put(3.5,1){$\bullet$}
      \put(3.8,1){$I_1X=M(y+r,y)$}
    \put(3.5,2){$\bullet$}
    \put(3.5,3){$\bullet$}
      \put(3.8,2){$Z=M(y+r,z)$}
      }
        \thicklines
        \put(0,-.5){
    \put(1,1){$\bullet$}
      \put(-1.3,1){$X=M(x,y)$}
    \put(1,2){$\bullet$}
    \put(1,3){$\bullet$}
      \put(-1.3,2){$Y=M(x,z)$}
      \put(-2.3,3){$I_2X=M(x,x+r)$}
      \put(3.8,3){$\Sig X=M(y+r,x+r)$}
             \put(1.1,1.1){\qbezier(0,0)(0,.5)(0,1)
             \qbezier(0,.7)(0,.5)(-.13,.4)
             \qbezier(0,.7)(0,.5)(.13,.4)}
             \put(1.1,2.1){\qbezier(0,0)(1,0)(2.5,0)
             \qbezier(1.7,0)(1.5,0)(1.4,0.13)
             \qbezier(1.7,0)(1.5,0)(1.4,-0.13)
             }
             \put(3.6,2.1){\qbezier(0,0)(0,.5)(0,1)
             \qbezier(0,.7)(0,.5)(-.13,.4)
             \qbezier(0,.7)(0,.5)(.13,.4)}
             }
\end{picture}}
}}
\caption{Basic positive triangle (Equation \ref{eq: basic up-down-up triangle})}
\label{fig updown}
\end{center}
\end{figure}

If $X=M(x,y)$ and $Y=M(w,y)$ where $y-r<x<w<y+r$ we get another distinguished triangle
\begin{equation}\label{eq: basic down-up-down triangle}
        M{(x,y)}\xrarrow{1}M{(w,y)}\xrarrow{-1}M{(w,x+r)}\xrarrow{1}\Sig M{(x,y)}
\end{equation}
which we call the \emph{basic negative triangle}. Here, the second morphism is $-1$ times a basic morphism. These points are plotted in Figure \ref{fig downup}

\begin{figure}[htbp]
\begin{center}
%
{
\setlength{\unitlength}{1cm}
{\mbox{
\begin{picture}(5,3)
    \thinlines
    \put(0,-.5){
      \qbezier(3.3,.8)(4,1.5)(4.5,2)
    \put(3.5,1){$\bullet$}
      \put(3.8,1){$I_1X=M(y+r,y)$}
\put(0,-.5){
      \qbezier(0.2,2.2)(1,3)(1.3,3.3)
    \put(3.5,3){$\bullet$}
      \put(1.2,3.5){$Z'=M(w,x+r)$}}
      }
        \thicklines
        \put(0,-.5){
    \put(1,1){$\bullet$}
      \put(-1.3,1){$X=M(x,y)$}
    \put(2.2,1){$\bullet$}
      \put(1.5,.5){$Y'=M(w,y)$}
      \put(0,-.5){
    \put(2.2,3){$\bullet$}
    \put(1,3){$\bullet$}
      \put(-2.3,3){$I_2X=M(x,x+r)$}
      \put(3.8,3){$\Sig X=M(y+r,x+r)$}}
             \put(2.3,2.6){\qbezier(0,0)(.5,0)(1.2,0) 
             \put(.25,0){\qbezier(.7,0)(.5,0)(.4,0.13) 
             \qbezier(.7,0)(.5,0)(.4,-0.13)
             }}
                          \put(1.1,1.1){\qbezier(0,0)(.5,0)(1.2,0)
             \put(.25,0){\qbezier(.7,0)(.5,0)(.4,0.13) 
             \qbezier(.7,0)(.5,0)(.4,-0.13)}
             }
             \put(2.3,.6){\qbezier(0,0.5)(0,1)(0,2)
             \put(-0.6,1){$-1$}
             \qbezier(0,1.7)(0,1.5)(-.13,1.4) 
             \qbezier(0,1.7)(0,1.5)(.13,1.4)}
             }
\end{picture}}
}}
\caption{Basic negative triangle (Equation \ref{eq: basic down-up-down triangle})}
\label{fig downup}
\end{center}
\end{figure}
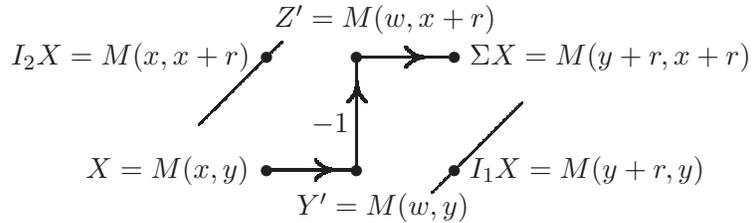

A third example is given by $X=M(x,y)$ and $Y=M(z,y)$ where $x<w<y+r$ and $y<z<x+r$. Then we get the distinguished triangle
\begin{equation}\label{eq: third distinguished triangle}
        M{(x,y)}\xrarrow{1}M{(z,y)}\xrarrow{\binom{-1}1}M(w,x+r)\oplus M(y+r,z)\xrarrow{(1,1)} \Sig M{(x,y)}
\end{equation}

Thus a basic positive triangle is distinguished when the morphism are basic morphisms and a basic negative triangle is distinguished when two of the morphisms are basic and the third is negative a basic morphism. Note that when we apply the functor $\Sig $ to a basic positive triangle it has the shape of a basic negative triangle but has positive signs and is thus \emph{not} a distinguished triangle unless $\kk$ has characteristic 2. This comes from the general fact that if a distinguished triangle remains distinguished when the sign of one or three of the morphisms is changed, then either the ground field has characteristic 2 or the triangle splits (is a direct sum of triangle with one term equal to zero).

Similarly, $\Sig $ changes a basic negative triangle into a triangle with the shape of a basic positive triangle, again with the wrong sign. We will describe more general triangles later using the hyperbolic plane. 

%
%



\section{Continuous orbit category} 

The \emph{continuous orbit category} $\cC_{r,s}$ is defined to be the  orbit category of the (doubled) continuous derived category $\cD_r\double$ under a triangulated automorphism $\ul F_s$. However, it is easier to see the triangulated structure of this category if we first take the (completed) orbit category of the ``doubled'' category $\cB_{\le r}\double$ under the exact automorphism $F_s$, show that the result is again a Frobenius category, then stabilize.

We will see later that $\cC_{r,s}$ has a cluster structure only for certain values of $\frac{r}{s}$, namely, $\frac{r}{s}=\frac{n+1}{n+3}$ and $\frac{r}{s}=1$. Thus the name ``continuous cluster category'' is only appropriate for those values of $r,s$. Otherwise, $\cC_{r,s}$ will be called a \emph{continuous orbit category.}

When we take the orbit category of the doubled Frobenius category $\cB_{\le r}\double$, we need to complete the category since, otherwise, the endomorphism ring of an indecomposable object will be the polynomial ring $\kk [u]$ which is not local. In the completed orbit category, the endomorphism ring of any indecomposable object will be the local ring $S=\kk [[u]]$ and $\Hom(X,Y)$ will be a free $S$-module of rank 1 for any two indecomposable objects $X,Y$. We will show that the completed orbit category $\cB_{\le r}/F_s^\wedge$ is equivalent to the continuous Frobenius category $\cF_{r\pi/s}$ constructed in \cite{IT09}. In the notation of \cite{IT09}, $u=\sqrt t$.


\subsection{Construction of completed orbit category}

For any $s\ge r$, we want to take an orbit category of $\cB_{\le r}$ with respect to a functor $F_s$ which takes $M(x,y)$ to $M(y+s,x+s)$. However, the obvious functor of this form is not strictly triangulated since it takes positive triangles to negative triangles. To remedy this we will ``double'' the category. For any additive category $\cE$, let $\cE\double$ be the category whose objects are ordered pairs $(X_0,X_1)$ of objects in $\cE$ and whose Hom sets are $\cE\double((X_0,X_1),(Y_0,Y_1))= \cE(X_0\oplus X_1,Y_0\oplus Y_1)$. Then $\cE\double$ is equivalent to the full subcategory of objects of the form $(X,0)$ since $(X_0,X_1)\cong (X_0\oplus X_1,0)$. Therefore, $\cE\double$ is always equivalent to $\cE$.

In the case at hand, $\cB_{\le r}\double$ is the strictly additive $\kk$-category with exactly two objects in every isomorphism class of indecomposable objects: $M(x,y)_+$ and $M(x,y)_-$. The advantage of this is that we can take different two-way approximation sequences for these two objects:
\begin{equation}\label{two-way approximations with parity}
\xymatrixrowsep{10pt}\xymatrixcolsep{25pt}
\xymatrix{
M(x,y)_\e\ \ \ar@{>->}[r]^(.3){\binom{\e}{-\e}} &
	M(y+r,y)_\e\oplus M(x,x+r)_\e\ar@{->>}[r]^(.58){(1,1)} &
	M(y+r,x+r)_\e
	}
\end{equation}
For $\e=+$, this agrees with \eqref{positive two-way approximation} and, in the stable category of $\cB_{\le r}\double$, this gives positive distinguished triangles for $M(x,y)_+$. For $\e=-$, the sign is reversed and we get negative distinguished triangles in the stable category starting with $M(x,y)_-$. 

Let $F_s$ be the strictly additive $\kk$-automorphism of $\cB_{\le r}$ given on indecomposable objects by
\[
	F_s \left(M(x,y)_\e\right)=M(y+s,x+s)_{-\e}
\]
and on morphism by taking basic morphisms to basic morphisms. Then $F_s$ takes the chosen two-way approximation sequence for $M(x,y)_\e$ to the two-way approximation sequence for $M(y+s,x+s)_{-\e}$ and therefore gives a strictly triangulated automorphism of the stable category. This stable category is isomorphic to the double of the continuous derived category $\cD_r$, so we denote it by $\cD_r\double$ and we define the triangulated structure of this doubled category to match that of the stable category of $\cB_{\le r}\double$ so that, e.g., one triangle starting with $M(x,y)_+$ and corresponding to \eqref{eq: third distinguished triangle} is
\[
       M{(x,y)_+}\xrarrow{1}M{(z,y)}\xrarrow{\binom{-1}1}M(w,x+r)\oplus M(y+r,z)\xrarrow{(1,1)} M{(y+r,x+r)_+}
\]
where the three terms in the middle can have any parity.

\begin{defn}
We define the \emph{completed orbit category} $\cB_{\le r}\double/F_s^\wedge$ to be the category with the same object set as $\cB_{\le r}\double$ but with Hom sets given by
\[
	\cB_{\le r}\double/F_s^\wedge(X,Y)=\prod_{n\in \ZZ}\cB_{\le r}\double(X,F_s^nY)
\]
with composition given by $(g_j)(f_k)=(h_n)$ where 
\[
	h_n=\sum_{j+k=n} F_s^kg_j\circ f_k
\]
This is a finite sum for each $n$ since $(\exists N\in \ZZ)(\forall j,k\le N)$ $g_j=0$ and $f_k=0$.
\end{defn}

For each $M(x,y)_\e$ in $\cB_{\le r}\double$, consider the basic morphism 
\[
	u: M(x,y)_\e\to F_s M(x,y)_\e=M(y+s,x+s)_{-\e}
\]
This is a natural transformation since any composition of basic morphisms is a basic morphism by definition. Given $M(x,y)_\e$ and $M(a,b)_{\e'}$ in $\cB_{\le r}\double$, let $m\in \ZZ$ be minimal so that there is a nonzero morphism
\[
	f_0:M(x,y)_\e\to F_s^m M(a,b)_{\e'}
\]
which we take to be basic. For example, if $m$ is even, then $x\le a+2s m, y\le b+2s m$ but either $x+s> b+2s m$ or $y+s>a+2s m$ or both. Then any morphism $f:M(x,y)_\e\to M(a,b)_{\e'}$ in the completed orbit category is given by a unique infinite sum of the form
\[
	f=\sum_{k\ge 0}a_k u^k f_0
\]
where $a_k\in \kk $ for all $k\ge0$. Composition of morphisms multiplies these coefficients as if they were elements of the ring $S=\kk [[u]]$. Therefore, we have:

\begin{prop} The completed orbit category $\cB_{\le r}\double/F_s^\wedge$ is an additive $S$-category. Given any two indecomposable objects $X,Y$ of $\cB_{\le r}\double/F_s^\wedge$, the Hom set $\cB_{\le r}\double/F_s^\wedge(X,Y)$ is a free $S$-module generated by one basic morphism, call it $f_{XY}$, and composition with a basic morphism $f_{YZ}:Y\to Z$ gives $f_{YZ}f_{XY}=u^n f_{XZ}$ for some $n\ge0$. In particular, the endomorphism ring of any indecomposable object is a local ring.\qed
\end{prop}

We leave it as an exercise for the reader to show that $n=0,1$ or $2$. Also $n=n(XYZ)$ is a cocycle in the sense that $n(XYZ)-n(XYW)+n(XZW)-n(YZW)=0$ for all indecomposable $X,Y,Z,W$. This structure will be explored in another paper \cite{IT12}.

Although $\cB_{\le r}\double/F_s^\wedge$ is an $S$-category we prefer to consider it as a category over $R=\kk [[t]]$ where $t=u^2$. As an $R$-module, each Hom set (between indecomposable objects $X,Y$) is freely generated by two morphisms $f_{XY}$ and $uf_{XY}$.


\subsection{Review of continuous Frobenius categories}\label{ss: review of Frobenius}

For any discrete valuation ring $R$ with unique maximal ideal $(t)$ and any positive real number $c\le\pi$ another Frobenius category $\cF_c$ was constructed in \cite{IT09}. We will review this constuction and, in the case when $R=\kk [[t]]$ and $c=\frac{r\pi}s$, we will show that $\cF_c$ is equivalent to the completed orbit category $\cB_{\le r}\double/F_s^\wedge$. We conclude that $\cB_{\le r}\double/F_s^\wedge$ is also a Frobenius category.

Let $\cP_{S^1}$ be the strictly additive $R$-category whose indecomposable objects are projective representations $P_{[x]}$ of the circle $S^1=\RR/2\pi\ZZ$ generated at one point $[x]=x+2\pi\ZZ$. For every nonnegative real number $\a$ we have a morphism $e^\a: P_{[x]}\to P_{[x+\a]}$ with composition given by $e^\a e^\b=e^{\a+\b}:P_{[x]}\to P_{[x+\a+\b]}$. The morphism sets $\cP_{S^1}(P_{[x]},P_{[y]})$ are free $R$-modules of rank 1 with generator equal to $e^\a$ where $\a$ is the smallest nonnegative real number so that $[x+\a]=[y]$. We call this the \emph{basic morphism} from $P_{[x]}$ to $P_{[y]}$. All morphisms of the form $e^{2\pi}$ are multiplication by the uniformizer $t\in R$ by definition. See \cite{IT09} for details. (See also the construction of the \emph{big loop} in \cite{vR}.)

Let $\cF$ be the category of all pairs $(V,d)$ where $V$ is an object of $\cP_{S^1}$ and $d$ is an endomorphism of $V$ so that $d^2=\cdot t$ (multiplication by $t$). Morphisms $(V,d)\to (W,d)$ are defined to be morphisms $f:V\to W$ in $\cP_{S^1}$ so that $df=fd$. An example of an object of $\cF$ is given, for any $x\le y\le x+2\pi \in\RR$, by
\[
	E(x,y):=\left(P_{[x]}\oplus P_{[y]},\mat{0 & e^{x+2\pi-y}\\e^{y-x} &0}\right)
\]
Note that $E(x,y)=E(x+2\pi,y+2\pi)$ and $E(x,y)\cong E(y,x+2\pi)$. Morphisms between these objects are given by $2\times 2$ matrices which can always be written as the sum of a diagonal morphism and a counter diagonal morphism. For example $d$ is always a counterdiagonal endomorphism of any object and $d^2=\cdot t$ is a diagonal endomorphism.

The following theorem in proved in \cite{IT09} and generalized in \cite{IT12} using the results of \cite{IT09}.

\begin{thm}\cite{IT09} Let $\cF$ be the category of all pairs $(V,d)$ where $V$ is an object of $\cP_{S^1}$ and $d$ is an endomorphism of $V$ so that $d^2=\cdot t$ with morphisms those commuting with $d$. Then $\cF$ is a Frobenius category with exact sequences defined to be sequences 
\[
	\xymatrixrowsep{10pt}\xymatrixcolsep{12pt}
\xymatrix{
(A,d)\ \ar@{>->}[r]^f &
	(B,d)\ar@{->>}[r]^g &
	(C,d)
	}
\]
so that $\xymatrixrowsep{10pt}\xymatrixcolsep{12pt}
\xymatrix{
A\ \ar@{>->}[r]^f &
	B\ar@{->>}[r]^g &
	C
	}
$ is a split exact sequence in $\cP_{S^1}$. The projective injective objects are those of the form $(P\oplus P,d_P)$ where $P$ is any object of $\cP_{S^1}$ and $d_P=\mat{0 &t\\1 &0}$. Furthermore, $\cF$ is a Krull-Schmidt category with indecomposable objects isomorphic to $E(x,y)$ for $x\le y\le x+2\pi$.
\end{thm}


\begin{cor}\cite{IT09} Given a positive real number $c\le\pi$, let $\th=\pi-c$ and let $\cF_c$ denote the full subcategory of $\cF$ generated by the indecomposable objects $E(x,y)$ where $x+\th\le y\le x+2\pi-\th$. Then $\cF_c$ is a Frobenius category with exact sequences defined to be those in $\cF$ with all objects in $\cF_c$ and with projective-injective objects given by $E(x,y)$ where $|y-x-\pi|=c$.
\end{cor}

The stable category of $\cF_c$ is defined to be $\cC_c$. The \emph{continuous cluster category} is defined to be $\cC_\pi$. To compare the categories $\cB_{\le r}\double/F_s^\wedge$ and $ \cF_{r\pi/s}$ we use the following notion. The \emph{continuous degree} of any nonzero morphism $f:M(x,y)\to M(a,b)$ is defined to be $cdeg\,f:=a+b-x-y$. This has the property that $cdeg\,(fg)=cdeg\,f+cdeg\,g$ if $f,g$ are nonzero morphisms between indecomposable objects of $\cB_{\le r}$. Any nonzero morphism $f:P_{[x]}\to P_{[y]}$ in $\cP_{S^1}$ has the form $f=re^{\a}$ where $x+\a-y\in 2\pi\ZZ$ and $r$ is a unit in $R=\kk [[t]]$. The \emph{continuous degree} of $f$ is then defined to be $cdeg\,f=\a$. The \emph{continuous degree} of any nonzero diagonal or counterdiagonal morphism $E(x,y)\to E(a,b)$ is defined to be the sum of the continuous degrees of the diagonal or counterdiagonal entries. For example, $cdeg\,d=2\pi$.


\subsection{The equivalence $\hat G$} There is an equivalence of categories \[
\hat G:\cB_{\le r}\double/F_s^\wedge\xrarrow\cong \cF_{r\pi/s}
\]
induced by $G:\cB_{\le r}\double\to\cF_{r\pi/s}$, the strictly additive $\kk$-linear functor defined on indecomposable objects and basic morphisms as follows. 

On indecomposable objects $M(x,y)_\e$ of $\cB_{\le r}\double$, $G$ is defined by:
\[
	GM(x,y)_\e=\begin{cases} E(\frac\pi sx,\frac\pi sy+\pi) & \text{if } \e=+\\
   E(\frac\pi sy-\pi,\frac\pi sx) & \text{if } \e=-
    \end{cases}
\]
For $M(x,y)_\e$ in $\cB_{\le r}\double$ we have $x-r\le y\le x+r$. So, $x+\th\le y+\pi\le x+2\pi-\th$ where $\th=s-r$ and $E(x,y+\pi)\cong E(y-\pi,x)$ is an object of $\cF_{r\pi/s}$.

For every basic morphism $f:M(x,y)_\e\to M(a,b)_{\e'}$ in $\cB\double$ with $cdeg\,f:=a+b-x-y$, we define $Gf:GM(x,y)_\e\to GM(a,b)_{\e'}$ to be the unique morphism of the form $Gf=d^ng_0=g_0d^n$ where $g_0:GM(x,y)_\e\to GM(a,b)_{\e'}$ is a basic morphism and $n\in\ZZ_{\ge0}$ so that $d^ng_0$ has continuous degree equal to $\frac\pi s\,cdeg\, f$. A formula for $Gf$ is given by:
\[
	Gf=\mat{0&1\\1&0}^{(\e'=-)}\mat{e^{a-\frac\pi sx} & 0\\ 0 & e^{b-\frac\pi sy}}\mat{0&1\\1&0}^{(\e=-)}
\]
where $(\e=-)$ denotes the truth value of that statement (1 if true and 0 if false). In particular, $Gf$ is a diagonal morphism if $\e=\e'$ and a counterdiagonal morphism if $\e=-\e'$.

The formula clearly shows that $G$ is a continuous functor and it follows from the topology of covering spaces that $G$ is uniquely determined, as a continuous strictly additive $\kk$-linear functor, by its value on indecomposable objects (just as in Proposition \ref{uniqueness of continuous functors}).

\begin{thm} The functor $G:\cB_{\le r}\double\to\cF_{r\pi/s}$ induces an equivalence of categories:
\[
\hat G:\cB_{\le r}\double/F_s^\wedge\xrarrow\cong \cF_{r\pi/s}
\]
\end{thm}

\begin{proof} The proof uses the following diagram. 
\[
\xymatrix{
\cB_{\le r}\double\ar[d]\ar[r] &
	\cB_{\le r}\double/F_s^\wedge\ar[d]\ar[r]^(.6)\cong_(.6){\hat G} & \cF_{r\pi/s}\ar[d]
	\\
\cD_r\double \ar[r]& 
	\cD_r\double/\ul F_s\ar[r]^(.6)\cong& \cC_{r,s}
	}
\]
First, we will show that this diagram commutes. Commutativity of the left hand square is a tautology. All four categories have the same set of objects by definition and the four functors are the identity maps on these objects again by definition. The functors send basic morphisms to basic morphism. So, the left side of the diagram commutes.

For commutativity of the right hand square note first that $G=G\circ F_s$ by the following calculation.
\[
	GF_s M(x,y)_\e=GM(y+s,x+s)_{-\e}=\begin{cases} E(\frac \pi sy+\pi,\frac \pi sx+2\pi)=E(\frac \pi sy-\pi,\frac \pi sx) & \text{if } \e=-\\
   E(\frac \pi sx,\frac \pi sy+\pi) & \text{if } \e=+
    \end{cases}
\]
This equals $GM(x,y)_\e$. So, $G$ takes the same value on all objects in any $F_s$ orbit.

Given two indecomposable objects $M(x,y)_\e,M(a,b)_{\e'}$, let $n\in\ZZ$ be minimal so that there exist a nonzero morphism $f_0:M(x,y)_\e\to F_s^n M(a,b)_{\e'}$. Then a morphism \[
	f\in \cB_{\ge c}\double/F_s^\wedge(M(x,y)_\e,M(a,b)_{\e'})=\prod_{i\in\ZZ}\cB_{\ge c}\double(M(x,y)_\e,F_s^iM(a,b)_{\e'})
\]
is given by $f=\sum_{i=n}^\infty a_iu^{i-n}f_0$ where $u$ denotes the basic morphism $X\to F_s X$ for all objects $X$ and $a_i\in \kk $ for all $i\ge n$. Let $g_0:GM(x,y)_\e\to GM(a,b)_{\e'}=GF_s^nM(a,b)_{\e'}$ be the basic morphism. Then we define $\hat Gf$ to be the morphism
\[
	\hat Gf= vg_0+wdg_0 \in\cF_{r\pi/s}(GM(x,y)_\e, GM(a,b)_{\e'})
\]
where $v,w\in R=\kk [[t]]$ are given by $v=\sum_{k\ge0} a_{n+2k}t^k$ and $w=\sum_{k\ge0} a_{n+2k+1}t^k$. Since $Gu^2=t$, this gives an $R$-linear isomorphism
\[
	\hat G:\cB_{\ge c}\double/F_s^\wedge(M(x,y)_\e,M(a,b)_{\e'})\cong \cF_{r\pi/s}(GM(x,y)_\e, GM(a,b)_{\e'})
\]

To check that $\hat G$ is a functor, take another indecomposable $M(p,q)_{\e''}$ and let \[
f'=\sum b_ju^{j-m} f_1\in \cB_{\ge c}\double/F_s^\wedge(M(a,b)_{\e'}, M(p,q)_{\e''})
\]
where $f_1:M(a,b)_{\e'}\to F^m M(p,q)_{\e''}$ is the basic morphism with minimal $m\in\ZZ$. Then $\hat Gf'=r'g_1+s'dg_1$ where $r'=\sum b_{m+2k}t^k$, $s'=\sum b_{m+2k+1}t^k$ and $g_1:GM(a,b)_{\e'}\to GM(p,q)_{\e''}$ is the basic morphism. Then $g_1g_0=d^\el g_2$ where $g_2$ is basic (and $\el=0,1$ or $2$). This corresponds to $f_1f_0=u^\el f_2$ where $f_2:M(x,y)_\e\to F_s^{n+m-\el}M(p,q)_{\e''}$ is basic. Calculation shows:
\[
	f'  f=(r' +s'u) f_1(r+su)f_0=(r' +s'u)(r+su)u^\el f_2
\]
\[
	Gf' Gf=(r' +s'd) g_1(r+sd)g_0=(r' +s'd)(r+sd)d^\el g_2=G(f'f)
\]

Since the objects in an $F_s$ orbit in $\cB_{\ge c}\double/F_s^\wedge$ are canonically isomorphic, this induces an equivalence of categories $\hat G:\cB_{\le r}\double/F_s^\wedge\xrarrow\cong \cF_{r\pi/s}$ as claimed.
\end{proof}

\begin{cor}
The completed orbit category $\cB_{\ge c}\double/F_s^\wedge$ is a Frobenius category whose stable category is equivalent to the orbit category $\cC_{r,s}=\cD_r\double/\ul F_s$. So, $\cC_{r,s}\cong \cC_{r\pi/s}$.\qed
\end{cor}

\begin{rem}
In particular, when $r=s=\pi$, the orbit category $\cC_{r,s}\cong \cC_\pi$ is triangulated and we will show in Theorem \ref{C has cluster structure} that it has a cluster structure.
\end{rem}

\begin{rem}
$\cB_{\ge c}\double/F_s^\wedge$ is a topological $R$-category with space of indecomposable objects homeomorphic to the connected and oriented 2-fold covering space of the closed Moebius band. At the end of \cite{IT09} another topological construction is given which forms the disconnected and unoriented 2-fold covering space of the closed Moebius band. Thus there are two inequivalent topological triangulated structures on $\cC_{r,s}\cong \cC_{r\pi/s}$.
\end{rem}

\subsection{Embedding the cluster category of type $A_n$}

Suppose that $\frac rs=\frac{n+1}{n+3}$. Then we will construct an embedding of the cluster category of type $A_n$ into the continuous orbit category $\cC_{r,s}$. But first, we need the general description of the continuous orbit category $\cC_{r,s}=\cD_r/\ul F_s$ for any $0<r\le s$. For this we use the continuous degree of any morphism and $w=s-r$.

The indecomposable objects of $\cC_{r,s}$ are $E(x,y)$ where $x,y$ are real numbers so that $x+w<y<x+2s-w$ modulo the identification $E(x,y)=E(x+2ms,y+2ms)$ for any $m\in \ZZ$ and $E(x,y)\cong E(y,x+2s)$. Hom sets are given by $\cC_{r,s}(E(a,b),E(x,y))=\kk $ if $x,y$ can be chosen so that either
\begin{enumerate}
\item $a\le x <b-w$ and $b\le y<a+2s-w$ or
\item $a\le y <b-w$ and $b\le x+2s<a+2s-w$
\end{enumerate}
and the Hom set is zero otherwise. In other words, there is a nonzero homomorphism $E(a,b)\to X$ with $X$ indecomposable if and only if $X\cong E(x,y)$ where $(x,y)$ lies in the half-open rectangle
\[
	R_{(a,b)}=[a,b-w)\times [b,a+2s-w)
\]
This is because the upper left and lower right corners of this rectangle are the projective-injective objects $E(a,a+2s-w)$ and $E(b-w,b)$ respectively. The \emph{continuous degree} of any nonzero morphism $f:E(a,b)\to E(x,y)$ is given by $cdeg\,f=x+y-a-b$ in the first case and $cdeg\,f=2s +x+y-a-b$ in the second case. In particular, all nonzero morphism have continuous degree $<2c=2s-2w$. 

Composition of nonzero morphisms $f:X\to Y, g:Y\to Z$ in $\cC_{r,s}$ for $X,Y,Z$ indecomposable is given by multiplication of scalars provided that $\cC_{r,s}(X,Z)=\kk $ and $cdeg\,f+cdeg\,g<2c$. This is necessary since $cdeg(g\circ f)=cdeg\,f+ cdeg\,g$. The condition is also sufficient since it implies that $Y$ and $Z$ lie inside the rectangle $R_X$ in such a way that $Z$ is northeast of $Y$ and therefore the morphisms $f,g$ are entirely inside $R_X$. 

Figure \ref{Fig: composition in C} shows the example: $X=E(0,7s/6)$, $Y=E(2s/3,5s/3)$, $Z=E(4s/3,2s)\cong Z'=E(0,4s/3)$ with $w<s/3$. There are nonzero morphisms $f:X\to Y,g:Y\to Z,h:X\to Z$ with $cdeg\,f =7s/6$, $cdeg\,g=s$ and $cdeg\,h=s/6$. The composition $g\circ f$ is not equal to a multiple of $h$ since the unique point $(0,4s/3)$ in the half-open rectangle $R_X$ isomorphic to $Z$ is to the southwest of $Y$ and this is equivalent to the fact that $cdeg\,f+cdeg\,g\neq cdeg\,h$.

%
\begin{figure}[htbp]
\begin{center}
%
\setlength{\unitlength}{1cm}
\mbox{
\begin{picture}(6,4)
 \put(0,0.5){
    \thinlines
    \put(2.5,-0.5){\qbezier(0,0)(1,1)(3.5,3.5)
    \put(3.5,2.9){$y=x+w$}}
    \put(-1,1){\qbezier(0,0)(2,2)(2.6,2.6)
    \put(-1,1.8){$y=x+2s-w$}}
     \thicklines
     \put(-.1,-.1){$\bullet$}
     \put(-.5,-.2){$X$}
     \put(1.4,1.4){$\bullet$}
     \put(1,1.3){$Y$}
     \put(3.9,2.4){$\bullet$}
     \put(3.5,2.4){$Z$}
     \put(-.1,.4){$\bullet$}
     \put(-.55,.4){$Z'$}
     \put(0,0){ 
     \put(1.5,.3){$R_X$}
     \thicklines
     \put(0,0){\qbezier(0,0)(1,0)(3,0)
     \qbezier(0,0)(0,1)(0,2)
     \thinlines
     \put(-.05,2){
     \put(.2,0){.}
     \put(.4,0){.}
     \put(.6,0){.}
     \put(.8,0){.}
     \put(1,0){.}
     }
     \put(1.95,2){
     \put(.2,0){.}
     \put(.4,0){.}
     \put(.6,0){.}
     \put(.8,0){.}
     \put(1,0){.}
     }
     \put(.95,2){
     \put(.2,0){.}
     \put(.4,0){.}
     \put(.6,0){.}
     \put(.8,0){.}
     \put(1,0){.}
     }
     \put(2.95,.2){.}
     \put(2.95,.4){.}
     \put(2.95,.6){.}
     \put(2.95,.8){.}
     \put(2.95,1){.}
     \put(2.95,1.2){.}
     \put(2.95,1.4){.}
     \put(2.95,1.6){.}
     \put(2.95,1.8){.}
     }
     } 
     \put(1.5,1.5){ 
     \put(.2,1){$R_Y$}
     \thicklines
     \put(0,0){\qbezier(0,0)(1,0)(3,0)
     \qbezier(0,0)(0,1)(0,2)
     \thinlines
     \put(-.05,2){
     \put(.2,0){.}
     \put(.4,0){.}
     \put(.6,0){.}
     \put(.8,0){.}
     \put(1,0){.}
     }
     \put(1.95,2){
     \put(.2,0){.}
     \put(.4,0){.}
     \put(.6,0){.}
     \put(.8,0){.}
     \put(1,0){.}
     }
     \put(.95,2){
     \put(.2,0){.}
     \put(.4,0){.}
     \put(.6,0){.}
     \put(.8,0){.}
     \put(1,0){.}
     }
     \put(2.95,.2){.}
     \put(2.95,.4){.}
     \put(2.95,.6){.}
     \put(2.95,.8){.}
     \put(2.95,1){.}
     \put(2.95,1.2){.}
     \put(2.95,1.4){.}
     \put(2.95,1.6){.}
     \put(2.95,1.8){.}
     }
     } 
     } 
\end{picture}
} 
\caption{$Z'=E(0,\frac{4s}3)\in R_X=[0,\frac{7s}6-w)\times[\frac{7s}6,2s-w)$ is southwest of $Y=E(\frac{2s}3,\frac{5s}3)$. So, a nonzero morphism $X\to Z\cong Z'$ cannot factor through $Y$.}
\label{Fig: composition in C}
\end{center}
\end{figure}
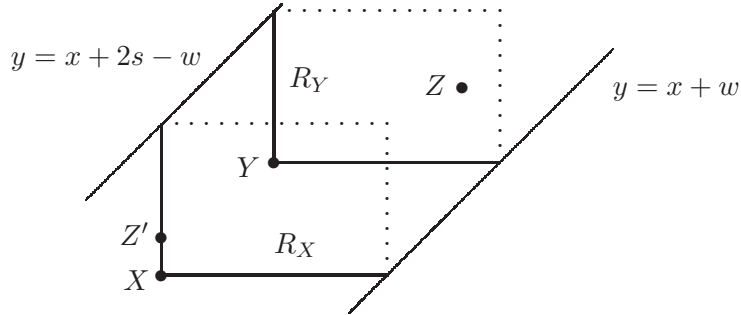

\begin{lem}
Let $0<\th< \pi$ and let $Z$ be any nonempty subset of $S^1$ with the property that, for all $z\in Z$, both $z+\th$ and $z-\th$ lie in $Z$. Let $\cF(Z)$ be the additive full subcategory of $\cF_c$, where $c=\pi-\th$, generated by objects $E(x,y)\in \cF_c$ where $x+\th\le y\le x+2\pi-\th$ and $x,y$ lie in $Z$. Then $\cF_c$ is a Frobenius category with projective-injective objects $E(x,x+\th)$ and $E(x+\th,x+2\pi)$. Consequently, the embedding $\cF(Z)\into \cF_c$ induces a strictly triangulated embedding of stable categories $\ul\cF(Z)\into \ul\cF_c\cong\cC_{r,s}$ whenever $\frac rs=\frac c\pi $.
\end{lem}

\begin{proof}
Given any exact sequence $(A,d)\cof (B,d)\onto (C,d)$ in $\cF_c$ with $(B,d)\in\cF(Z)$, the components of $B\cong A\oplus C$ are projective representations of $S^1$ generated at points of $Z$. Then $A,C$ also have this property. So, $(A,d),(C,d)$ lie in $\cF(Z)$. By Lemma \ref{lemma for subcategories of exact categories to be exact}, $\cF(Z)$ is an exact category. For any object in $\cF(Z)$, its projective cover and injective envelope in $\cF_c$ lies in $\cF(Z)$ by definition. So, $\cF(Z)$ is a Frobenius category. 

Since the projective-injective objects of $\cF(Z)$ are equal to the projective injective objects of $\cF_c$ which lie in $\cF(Z)$, the embedding $\cF(Z)\into \cF_c$ induced an embedding of stable categories $\ul\cF(Z)\into\ul\cF_c\cong\cC_{r,s}$. If we make compatible choices of two-way approximation sequences using the formulas described in \eqref{two-way approximations with parity} this embedding will be strictly triangulated.
\end{proof}

\begin{thm}\label{F(Z)=C(An) is embedded in Cc}
Suppose that $\frac rs=\frac{n+1}{n+3}$ and let $Z$ be any subset of $S^1$ of the form $Z=\{a,a+\th,\cdots,a+(n+2)\th\}$ where $\th=\frac{2\pi}{n+3}$. Then the stable category of $\cF(Z)$ is equivalent to the cluster category of type $A_n$ over the field $\kk$ and the embedding $\cF(Z)\into \cF_{r\pi/s}$ induces a triangulated embedding $\ul \cF(Z)\into \cC_{r,s}$.
\end{thm}

\begin{proof} Take the image of $Z$ under the standard embedding of $S^1$ to the unit circle in the plane. These points form the vertices of a regular $(n+3)$-gon. The indecomposable objects of $\ul\cF(Z)$ are $E(x,y)$ where $x,y$ are nonconsecutive vertices of this regular polygon. The sides of the polygon are zero since they correspond to projective-injective objects of $\cF(Z)$. By the description of morphisms and their composition in $\cC_{r,s}$ given above, morphisms between indecomposable objects are given by counterclockwise rotation of the corresponding chords as described in \cite{CCS06}. Therefore, $\cF(Z)$ is equivalent to the cluster category of type $A_n$. By the general theorem of Keller and Reiten \cite{KR}, this is a triangulated equivalence.
\end{proof}





\section{Clusters and laminations}\label{sec3} 

This section is devoted to the definition of a cluster in the continuous cluster category $\cC=\cC_\pi$ defined in Section \ref{ss: review of Frobenius}. There is a short digression into the concept of ``laminations'' as originally introduced by W. Thurston \cite{Th}. The algebraic purpose of this digression is to explain why a maximal compatible set of indecomposable objects in $\cC$ need not form a cluster. 

In this section we let $\cM$ be the set of all isomorphism classes of indecomposable objects of $\cC$ considered as a topological space with the quotient topology. Thus $\cM$ is an open Moebius band whose elements are unordered pairs of distinct points in $S^1$ and $Ind\,\cC$ is the two fold connected covering of $\cM$ whose elements are given by ordered pairs of distinct points on the circle.

We begin by defining ``compatibility'' for indecomposable objects of the continuous cluster category. This condition is equivalent to the condition that the corresponding geodesics in the hyperbolic plane do not cross. A \emph{lamination} can then be defined to be a closed subset of pairwise compatible elements of $\cM$. We define a \emph{cluster} to be a discrete maximal lamination. Since discrete implies closed, this is the same as a discrete maximal compatible subset of $\cM$.

\subsection{Compatibility} 

\begin{defn}
Two indecomposable objects $X,Y$ of the continuous cluster category $\cC=\cC_\pi$ are defined to be \emph{compatible} if either $\cC(X,Y)=0$ or $\cC(Y,X)=0$ or $X\cong Y$.
\end{defn}

To fully understand this definition, we need to interpret it in terms of the Moebius band $\cM$ and in terms of geodesics in the hyperbolic plane. We give statements without much proof since the proofs are very straightforward, following from the definitions. We use the asymmetric notation $E(x,y)=M(x,y-\pi)$ where $x<y<x+2\pi$.

\subsubsection{Geometric interpretation of compatibility in terms of $\cM$}

\begin{prop}\label{prop: interpretation of compatible in M}
The objects compatible with $X=E(a,b)$ are $E(x,y)$ and $E(x,y)'=E(y,x+2\pi)$ where either 
\begin{enumerate}
\item $a\le x<y\le b$ ($E(x,y)$ is ``southeast'' of $E(a,b)$) or
\item $x\le a<b\le y<x+2\pi$ ($E(x,y)$ is ``northwest'' of $E(a,b)$).\qed
\end{enumerate}
\end{prop}

Note that the notion of ``northwest'' and ``southeast'' depend on the choice of coordinates for $X$. However, the union of these two regions is independent of the choice of coordinates. The compatible regions form right triangles which are closed on the two short sides and open on the hypotenuse. The incompatible region is an open rectangle with $X,X'$ on opposite corners. A simple example is given in Figure \ref{Fig: compatibility in M} where $A$ is the northwest compatible region and $B$ is the southeast compatible region. Note that $\cC(Y,X)\neq0$ since morphisms go right and up. But $\cC(X,Y)\cong \cC(X,Y')=0$ since the morphism $X\to Y'$ factors through a point in the boundary (the upper left corner of $C$).

%
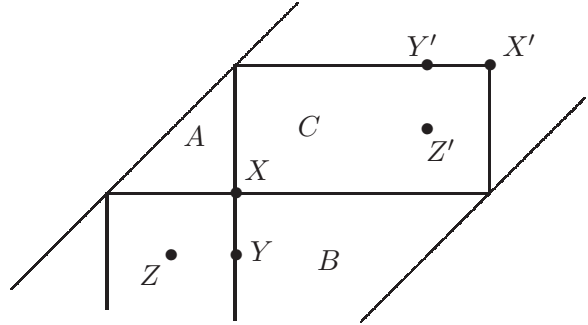
\begin{figure}[htbp]
\begin{center}
%
{
\setlength{\unitlength}{1in}
{\mbox{
\begin{picture}(4,1.7)
    \thinlines
    \put(2,.7){
      \qbezier(.33,-.67)(1,0)(1.5,.5) 
      \qbezier(-1.5,-.5)(-1,0)(0,1)
     \thicklines
  %
\qbezier(-1,0)(0,0)(1,0) 
\qbezier(-.33,-.66)(-.33,.33)(-.33,.66) 
\qbezier(1,.67)(-.1,.67)(-.33,.67) 
\qbezier(1,.67)(1,.3)(1,0) 
\qbezier(-1,0)(-1,-.4)(-1,-.6) 
\put(-.03,-.03){    
    \put(-.33,0){$\bullet$} 
    \put(-.33,-.33){$\bullet$} 
    \put(1,.67){$\bullet$}
    \put(1.1,.75){$X'$}
        \put(.67,.67){$\bullet$}
        \put(.6,.75){$Y'$}
        \put(.67,.33){$\bullet$}
        \put(.7,.2){$Z'$}
        \put(-.67,-.33){$\bullet$}
        \put(-.8,-.45){$Z$}
    }
    \put(.1,-.4){$B$}
    \put(-.6,.25){$A$}
    \put(0,.3){$C$}
    \put(-.28,.07){$X$}
    \put(-.25,-.35){$Y$}
      }
\end{picture}}
}}
\caption{$X=E(0,\frac{4\pi}3)$ is compatible with points in regions $A,B$ such as $Y=E(0,\pi)$ but not with points in the open rectangle $C$ such as $Z'=E(\pi,\frac{5\pi}3)$.}
\label{Fig: compatibility in M}
\end{center}
\end{figure}

\subsubsection{Geometric interpretation of compatibility in terms of geodesics}

In the notation $E(a,b)$, the points in our open Moebius band $\cM$ corresponds to pairs of distinct point on the circle $S^1=\RR/2\pi\ZZ$ which corresponds to geodesics in the circular model of the hyperbolic plane $\mathfrak h^2$. The pair of points $\{a,b\}$ represents the unique geodesic in $\mathfrak h^2$ converging to those two points on the \emph{circle at infinity}.

\begin{prop}\label{prop: interpretation of compatible in h}
Two indecomposable objects of $\cC$ are compatible if and only if the corresponding geodesics do not cross.\qed
\end{prop}

Figure \ref{Fig: compatibility in h} represents the same example as the one given in Figure \ref{Fig: compatibility in M}. Morphisms are given by angles which always go counterclockwise. Thus $\a,\b,\g$ are nonzero morphisms $\a:X\to Z,\b:Z\to X,\g:Y\to X$.
%
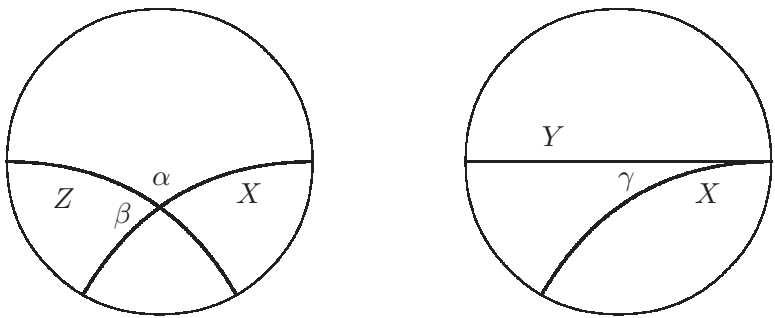
\begin{figure}[htbp]
\begin{center}
%
{
\setlength{\unitlength}{.8in}
{\mbox{
\begin{picture}(5,2)
    \thinlines
    \put(1,1){
 \qbezier(.71,.71)(1,.43)(1,0)
 \qbezier(.71,-.71)(1,-.43)(1,0)
 \qbezier(-.71,.71)(-1,.43)(-1,0)
 \qbezier(-.71,-.71)(-1,-.43)(-1,0)
  \qbezier(.71,.71)(.43,1)(0,1)
  \qbezier(.71,-.71)(.43,-1)(0,-1)
  \qbezier(-.71,.71)(-.43,1)(0,1)
  \qbezier(-.71,-.71)(-.43,-1)(0,-1)}
    \put(4,1){
 \qbezier(.71,.71)(1,.43)(1,0)
 \qbezier(.71,-.71)(1,-.43)(1,0)
 \qbezier(-.71,.71)(-1,.43)(-1,0)
 \qbezier(-.71,-.71)(-1,-.43)(-1,0)
  \qbezier(.71,.71)(.43,1)(0,1)
  \qbezier(.71,-.71)(.43,-1)(0,-1)
  \qbezier(-.71,.71)(-.43,1)(0,1)
  \qbezier(-.71,-.71)(-.43,-1)(0,-1)}
      \thicklines
 \put(1,1){     \qbezier(1,0)(0,0)(-.5,-.87)
 \put(.5,-.27){$X$}
 \qbezier(-1,0)(0,0)(.5,-.87)
 \put(-.7,-.3){$Z$} 
 \put(-.05,-.15){$\a$}
 \put(-.3,-.4){$\b$}
  }%
 \put(4,1){     \qbezier(1,0)(0,0)(-.5,-.87)
 \put(.5,-.27){$X$}
 \put(-1,0){\line(1,0){2}}
 \put(-.5,.1){$Y$}
 \put(0,-.15){$\gamma$}
  }%
\end{picture}}
}}
\caption{$X,Z$ are not compatible since they cross. But $X,Y$ are compatible since their intersection lies on the ``circle at $\infty$'' which is not in the set $\mathfrak h^2$.}
\label{Fig: compatibility in h}
\end{center}
\end{figure}

\subsection{Laminations} Following standard procedure, we would like to define a cluster to be a maximal set of pairwise compatible indecomposable objects of $\cC$. However, maximal compatible sets do not form a ``cluster structure'' because of the following fact.

\begin{prop}
Maximal compatible subsets of $\cM$ are closed.
\end{prop} 

Before we prove this, we want to point out the implication for cluster mutation. Suppose that $\cT$ is a maximal compatible set and $T\in \cT$. If $T$ is a limit point of $\cT$ then we cannot mutate $T$ by the obvious fact that, if we replace $T$ with a different object $T^\ast$ then the new set will not be closed. In other words, all limit points are ``frozen''. One example is the vertical line $\cT=0\times (-\pi,\pi)$, a maximal compatible set in which every point is a limit point. Thus, we will define a cluster to be a \emph{discrete} maximal compatible set.

\begin{proof} 
Suppose that $\cL$ is a maximal compatible subset of $\cM$. Then for any $T\in \cL$, $\cL$ is contained in the set $C(T)$ of all points in $\cM$ which are compatible with $T$. It follows from Proposition \ref{prop: interpretation of compatible in M} that $C(T)$ it a closed set. Since this holds for all $T\in \cL$ we have:
\[
	\cL\subseteq \bigcap_{T\in \cL}C(T)
\]
We claim that $\cL=\bigcap_{T\in \cL}C(T)$. Otherwise, we can add one element of $\bigcap_{T\in \cL}C(T)\backslash \cL$ to $\cL$ to make a larger compatible set contradicting the maximality of $\cL$. Therefore, $\cL=\bigcap_{T\in \cL}C(T)$. But any intersection of closed sets is closed. So, $\cL$ is closed.
\end{proof}

Now we make a brief digression to study maximal compatible subsets of $\cM$ in general.

\begin{defn}
A \emph{lamination} of the hyperbolic plane $\mathfrak h^2$ is a closed subset $C$ of $\mathfrak h^2$ and a homeomorphism $S\times \RR\cong C$, for some space $S$, which is equal to the exponential map on the second coordinate. In other words, this is a family of disjoint geodesics in $\mathfrak h^2$ parametrized by the topological space $S$.
\end{defn}

Since geodesics in the hyperbolic plane are disjoint if and only if the corresponding objects in $\cM$ are compatible, laminations of $\mathfrak h^2$ correspond to closed subsets of $\cM$ whose elements are pairwise compatible. We also call these sets ``laminations.'' To make sure that this interpretation is accurate, we need the following proposition.

\begin{prop}
A compatible subset $\cL$ of $\cM$ is closed if and only if the union $C$ of the corresponding geodesics in $\mathfrak h^2$ is closed.
\end{prop}

\begin{proof}
Suppose first that $\cL$ is closed in $\cM$ and $p$ is a point in $\mathfrak h^2$ which does not lie in $C$. For each point $x$ in the circle at $\infty$, there is a unique geodesic passing through $p$ and converging to $x$. This set of geodesics is compact, being homeomorphic to $S^1$ and the corresponding subset $K(p)$ of $\cM$ is also compact and disjoint from $\cL$ by construction. Since $K(p)$ depends continuously on $p$, there is a neighborhood $U$ of $p$ in $\mathfrak h^2$ so that $K(q)$ is disjoint from the closed set $\cL$ (as subset of $\cM$) for every $q\in U$. Then $U$ is disjoint from $C$ (as subsets of $\mathfrak h^2$) and we conclude that $C$ is closed in $\mathfrak h^2$.

Conversely, suppose that $\cL$ is not closed in $\cM$ and $E(x,y)$ is a limit point of $\cL$. Then the union of $\cL$ with $E(x,y)$ will be compatible since compatiblity is a closed condition in the sense that the set of all compatible pairs of points in $\cM$ is a closed subset of $\cM^2$. Therefore, the geodesic converging to $x,y$ on the circle at $\infty$ is disjoint from $C$ but the center points of some geodesics in $C$ converge to the center point of this new geodesic. Therefore $C$ is not closed in $\mathfrak h^2$ in this case.
\end{proof}

\subsubsection{Description of all maximal laminations}

The following theorem is essentially due to W. Thurston \cite{Th}.

\begin{prop} Any maximal lamination $\cL\subseteq\cM$ is locally embeddable in $\RR$ in the sense that every point in $\cL$ has a closed neighborhood $N$ which is homeomorphic to a compact subset $C$ of $\RR$. Conversely, all nonempty compact subsets $C$ of $\RR$ occur in this way. (There exists a maximal lamination $\cL$ and point in $\cL$ having a neighborhood homeomorphic to $C$.)
\end{prop}

\begin{proof}
Take any $E(a,b)\in \cL$. Then the coordinate difference function $(x,y)\mapsto x-y$ gives a continuous mapping from a neighborhood of $E(a,b)$ in $\cL$ to a neighborhood of $a-b$ in $\RR$. This mapping is 1-1 since any two points $E(x,y), E(x+d,y+d)$ with the same difference in coordinates are not compatible if $d<|x-y|$. Restriction of this map to a compact neighborhood $N=\cL\cap [a-\e,a+\e]\times[b-\e,b+\e]$ gives a homeomorphism of $N$ with a compact neighborhood of $a-b$.

Conversely, suppose that $C$ is any compact subset of $\RR$. By rescaling and adding disjoint points to $C$ if necessary, we may assume that $C$ lies in a closed interval $[a,b]$ where $\pi<a< b<3\pi/2$ and $a,b\in C$. Then the complement of $C$ in $[a,b]$ is a countable union of disjoint open intervals $(x_i,y_i)$. Then we can take $\cL$ to be the set with $E$-coordinates:
\[
	\cL=0\times \left((0,\pi]\cup C\cup [3\pi/2,2\pi)\right) \oplus \pi\times (\pi,a] \oplus b\times(b,3\pi/2]\oplus\bigoplus_i x_i\times (x_i,y_i]
\]
This is a maximal compatible set having a relatively open and closed subset $0\times C$ which is homeomorphic to $C$.
\end{proof}

\subsubsection{Clusters}

\begin{defn} A \emph{cluster} in $\cC$ is defined to be a discrete maximal lamination $\cT$ in $\cM$. By \emph{discrete} we mean that every $E(x,y)$ in $\cT$ has an open neighborhood that contains no other object of $\cT$.
\end{defn}

\begin{defn}
The \emph{standard cluster} is the set $\cT_0$ of all objects in $\cM$ with $E$-coordinates
\[
	\left(\frac{m\pi}{2^n},\frac{(m+1)\pi}{2^n}\right)
\]
for integers $n\ge0$ and $0\le m<2^{n+1}$. These objects are all nonisomorphic except for $E(0,\pi)\cong E(\pi,2\pi)$. If we reverse the parity of any of the objects in $\cT_0$ we will still consider it to be a standard cluster.
\end{defn}

\begin{prop}
The standard cluster $\cT_0$ is a cluster. In particular, clusters exist in $\cC_\pi$.
\end{prop}

\begin{proof} To see that the objects in the standard cluster are compatible, look at the interval given by the formula. Clearly any two of these with the same $n$ are disjoint and thus compatible and any two with different $n$ are either disjoint or one contains the other. So they are compatible. 

To show that this set is maximal, take any $E(x,y)$ which is compatible with all of these objects but is not in this set. The first coordinate $x$ must be an integer times $\pi/2^n$ for some $n$. Otherwise, for some sufficiently large $n$ we will have $m\pi/2^n<x<(m+1)\pi/2^n$ and $y$ outside this interval making the two geodesics cross. Thus $x=a\pi/2^n$ where $a=2k+1$ is odd. Then $x$ lies in the center of the interval $(k\pi/2^{n-1},(k+1)\pi/2^{n-1})$. To make $E(x,y)$ compatible with the corresponding object of $\cT_0$ we must have $|x-y|\le \pi/2^n$. 

If $|x-y|=\pi/2^m$ for any $m\ge n$ then $E(x,y)$ lies in $\cT_0$. Therefore, if $E(x,y)$ does not belong to the standard cluster, there exists an $m\ge n$ so that $\pi/2^m>|x-y|>\pi/2^{m+1}$. Suppose $y>x$. Then the object $E(x+\pi/2^{m+1},x+\pi/2^m)\in\cT_0$ crosses $E(x,y)$ which is a contradiction. The case $y<x$ is similar. Therefore $E(x,y)$ lies in $\cT_0$ making $\cT_0$ a maximal compatible set and therefore a cluster in $\cC$. 
\end{proof}

In the next section we will show that clusters in $\cC$ are unique up to triangulated automorphisms of $\cC$. This uniqueness theorem will make it easier to describe properties of arbitrary clusters.




\section{Automorphism of the continuous cluster category}\label{sec4}


\subsection{Linear subcategories and automorphisms of $\cC$} For each $z\in S^1=\RR/2\pi\ZZ$ let $\cL_z$ be the full subcategory of $\cC$ consisting of all indecomposable objects with one end at $z$. These are $E(z,x)$ and the isomorphic objects $E(z,x)'=E(x-2\pi,z)$ for all $z<x<z+2\pi$. We view these objects as lying on a vertical line or on a horizontal line. We call $\cL_z$ the \emph{linear subcategory} of $\cC$ at the end $z$. Note that $\cL_x\cap \cL_y=\{E(x,y),E(y,x+2\pi)\}$ has one object up to isomorphism since $E(x,y)\cong E(y,x+2\pi)$.

\begin{lem}\label{lem: characterization of linear subcategories}
A full subcategory $\cL$ of $\cC$ is equal to a linear subcategory $\cL_z$ for some $z$ if and only if it satisfies the following conditions.
\begin{enumerate}
\item $\cL$ has an infinite number of objects, and they are all indecomposable.
\item Any two objects $X,Y\in\cL$ are compatible and $\cC(X,Y)\oplus\cC(Y,X)\neq0$.
\item $\cL$ is maximal with the above two properties.
\end{enumerate}
\end{lem}

\begin{proof}
Any linear subcategory $\cL_z$ satisfies these three properties. Conversely, suppose that $\cL$ is a full subcategory of $\cC$ satisfying the above properties. Let $X=E(x,y)\in \cL$. Condition (2) implies that each object of $\cL$ shares at least one end ($x$ or $y$) with $X$. By Condition (1) there are infinitely many objects in $\cL$ with one of these ends, say $x$. To share one end with such an infinite family, it must be that every object of $\cL$ has one end at $x$. By Condition (3), $\cL=\cL_x$. 
\end{proof}

\begin{prop} For any $\kk$-linear automorphism $\Phi$ of $\cC$, there is a unique orientation preserving homeomorphism $\f$ of $S^1$ so that $\Phi E(x,y)\cong E(\f (x),\f(y))$.
\end{prop}

\begin{proof}
The image $\Phi(\cL_z)$ of any linear subcategory $\cL_z$ satisfies the conditions of the lemma above. It follows that $\Phi(\cL_z)=\cL_x$ for some $x\in S^1$. Since $\Phi^{-1}$ is also an automorphism, $\Phi$ induces a permutation $\f$ of the elements of $S^1$ so that $\Phi(\cL_z)=\cL_{\f(z)}$. 

It only remains to show that $\f$ is an orientation preserving homeomorphism of $S^1$. But this is equivalent to showing that $\f$ preserves the cyclic order of the elements of $S^1$. So, suppose that $x,y,z$ are in cyclic order: $x<y<z<x+2\pi$. Then the unique objects $Z=E(x,y)\in \cL_x\cap\cL_y$, $X=E(y,z)\in \cL_y\cap \cL_z$ and $Y=E(x,z)\in\cL_x\cap \cL_z$ map to each other in the reverse cyclic order: $Z\to Y\to X\to Z$ and there are no nonzero morphisms $X\to Y,Y\to Z$ or $Z\to X$. Applying $\Phi$ we obtain nonzero maps $\Phi Z\to\Phi Y\to \Phi X\to \Phi Z$ showing that $\f(x),\f(y),\f(z)$ are in correct cyclic order. So, $\f$ is an orientation preserving homeomorphism of $S^1$ as claimed.
\end{proof}

The above proposition implies that we have a group homomorphism from $Aut(\cC)$, the group of strictly additive strictly triangular automorphisms of $\cC$ to the group $Homeo_+(S^1)$ of all orientation preserving homeomorphism $\f$ of $S^1$. This homomorphism is surjective since, for any $\f\in Homeo_+(S^1)$, there is a strictly additive, strictly triangular automorphism $\Phi_\f$ of $\cC$ given on indecomposable object by $\Phi E(x,y)=(\f(x),\f(y))$ and on morphisms by sending basic morphisms to basic morphisms and extending linearly. This proves the following.

\begin{cor}
There is a split surjective group homomorphism
\[
	Aut(\cC)\to Homeo_+(S^1)
\]
whose kernel consists of all automorphisms of $\cC$ which send each object to an isomorphic object.
\end{cor}


\subsection{Equivalence of clusters}

In this subsection we will prove the following theorem which we interpret to mean that all clusters are equivalent and all objects in all clusters are equivalent.

\begin{thm}[Equivalence of clusters]\label{thm: equivalence of clusters}
For any two clusters $\cT_1,\cT_2$ in $\cC_\pi$ and any two objects $T_1\in\cT_1$, $T_2\in\cT_2$ there is a $\f\in Homeo_+(S^1)$ so that $\Phi_\f(\cT_1)\cong\cT_2$ and $\Phi_\f(T_1)=T_2$.
\end{thm}

Equivalently, we will show that, for any object $T_0$ of any cluster $\cT$ there is an automorphism $\Phi$ of $\cC$ which sends the standard cluster $\cT_0$ to $\cT$ and sends the object $E(0,\pi)$ to $T_0$. The automorphism will be $\Phi=\Phi_\f$ where $\f\in Homeo_+(S^1)$ will be given only on rational points of $S^1$ using the following lemma. Another lemma will show that the technical condition is automatically satisfied.

\begin{lem}\label{technical lemma for uniform continuity} Let $Q\subseteq S^1$ be the set of all points of the form $a\pi/2^n$ where $a,n\in\ZZ$ and $n\ge0$. Let $\f:Q\to S^1$ be a cyclic order preserving monomorphism. Then $\f$ extends uniquely to an orientation preserving homeomorphism of $S^1$ if and only if it satisfies the following condition.
\[
	(\forall \e>0)(\exists n\in\NN)(\forall a\in\ZZ)|\f(a\pi/2^n)-\f((a-1)\pi/2^n)|<\e
\]
\end{lem}

\begin{proof}
In words, the condition says that $\f$ is uniformly continuous. Any continuous function on a compact metric space, such as $S^1$, is uniformly continuous and any uniformly continuous mapping from a dense subset, such as $Q$, to a complete metric space extends uniquely to a continuous function on the whole space. The extended map lies in $Homeo_+(S^1)$ since it is cyclic order preserving.
\end{proof}

The set $Q$ is a union of an increasing sequence of finite subsets $Q_0\subset Q_1\subset Q_2\subset\cdots$ where $Q_n$ consisting of all points of the form $a\pi/2^n$. We will construct the function $\f$ on $Q$ as a union of a sequence of cyclic order preserving mappings $\f_n:Q_n\to S^1$, with $\f_n=\f_{n+1}|Q_n$. We do not assume the technical condition in the above lemma.

Note that any object of the standard cluster $\cT_0$ has endpoints in some $Q_n$. In fact, the rational cluster category $\cX$ is a union of finite subcategories $\cX_n$ consisting of objects with both ends in $Q_n$ and any cyclic order preserving monomorphism $\f_n:Q_n\to S^1$ gives an embedding $\Phi_n:\cX_n\to \cC$. To see this, note that $\f_n$ extends to some $\overline\f_n\in Homeo_+(S^1)$ and the restriction to $\cX_n$ of the induced automorphism of $\cC$ does not depend on the choice of extensions. The condition $\f_n=\f_{n+1}|Q_n$ will imply that $\Phi_n=\Phi_{n+1}|\cX_n$. So, $\Phi_m(\cT_0\cap \cX_n)$ will be independent of $m\ge n$.

\begin{lem}\label{second lemma for equivalence of clusters}
Let $\f_n:Q_n\to S^1$ be as above. If $\cT=\bigcup_{n\ge0} \Phi_n(\cT_0\cap \cX_n)$ is a cluster then the union $\f=\cup \f_n$ satisfies the technical condition in Lemma \ref{technical lemma for uniform continuity} and therefore extends to some $\overline\f\in Homeo_+(S^1)$ so that $\Phi_{\ov\f}(\cT_0)=\cT$.
\end{lem}

\begin{proof}
For any $\e>0$, let $U_\e$ be the set of all points in the open Moebius band which are within $\e$ of the boundary. Then the complement of $U_\e$ is a compact set and therefore contains only a finite number of elements of the cluster $\cT$. This finite set is contained in $\Phi_m(\cT_0\cap \cX_m)$ for some $m$. But this implies that for any $n>m$, $\Phi_n(\cT_0\cap(\cX_n-\cX_m))\subset U_\e$. But $\cT_0\cap (\cX_n-\cX_m)$ contains all objects of the form $E((a-1)\pi/2^n,a\pi/2^n)$. The statement that $\Phi_n$ sends this object into $U_\e$ is exactly the statement that $\f_n$ satisfies the technical condition in Lemma \ref{technical lemma for uniform continuity}.
\end{proof}

\begin{proof}[Proof of Theorem \ref{thm: equivalence of clusters}]
Take any cluster $\cT$ and any object $T_0\in\cT$. Then $T_0=E(x_0,x_1)$ for some $x_0<x_1<x_0+2\pi$. Let $\f_0:Q_0=\{0,\pi\}\to S^1$ be given by 
$
 {\f_0(j\pi)=x_j}
$ for $j=0,1$. 
Then $\Phi_0$ sends the object $E(0,\pi)\in \cT_0\cap \cX_0$ to $T_0$.

The other objects of $\cT$ are, up to isomorphism, contained in the two sets $A,B$ where
\[
	A=\{E(y_0,y_1)\ |\ x_0\le y_0<y_1\le x_1\},\quad B=\{E(y_0,y_1))\ |\ x_1\le y_1<y_0+2\pi\le x_0+2\pi\}
\]
To avoid the words ``up to isomorphism'' in the rest of this proof we assume that $\cT\subseteq A\cup B'$ where $B'=\{E(y_1,y_0+2\pi)\}$. Let $Y=E(y_0,y_1)\in A$ be the point in $\cT\cap A$ closest to $T_0$ in the 1-norm. Thus $d=|y_0-x_0|+|y_1-x_1|$ is minimal nonzero.

\ul{Claim 1}. Either $y_0=x_0$ or $y_1=x_1$.

If not then $x_0< y_0<y_1< x_1$. Then the point $Z=E(x_0,y_1)\in A$ is compatible with every point in $B$ and with every point in $A$ which is compatible with $Y$ and which has a distance of at least $d$ from $T_0$. So, $Z$ is compatible with every point in $\cT$. But $Z$ cannot be in $\cT$ since the distance from $X$ to $Z$ is less than $d$. This contradicts the maximality of $\cT$ proving the claim.

By symmetry we assume that $x_0= y_0$ and $y_1< x_1$. Since $x_0<y_1<x_1$ we change the notation to $y_1=x_{1/2}$. Then the set of points in $A$ compatible with $Y=E(x_0,x_{1/2})$ consists of $T_0$ and the two sets
\begin{equation}\label{equation for Aj}
	A_j=\{E(a,b)\ |\ x_{(j-1)/2}\le a<b\le x_{j/2}\}
\end{equation}
for $j=1,2$. Therefore,
\begin{equation}\label{partition of A}
	\cT\cap A\subseteq A_1\cup A_2\cup \{T_0\}
\end{equation}
Note that the triangular regions $A_1,A_2,B'$ are uniquely determined by their vertices which are $E(x_0,x_{1/2}), E(x_{1/2},x_1), E(x_1,x_0+2\pi)= T_0'$, respectively.

\ul{Claim 2}. The vertex $E(x_{1/2},x_1)$ of $A_2$ lies in $\cT$.

This follows from the maximality of $\cT$ since every element of $A_1\cup A_2\cup B'\cup \{T_0\}$ is compatible with $E(x_{1/2},x_1)$. 

The following claim follows from the analogous argument applied to the triangular region $B'$. (Take the point in $B'\cap \cT$ which is closest to but not equal to its vertex $T_0'$, etc.) We use the notation $x_2=x_0+2\pi$ (which is equal to $x_0$ as an element of $S^1$).

\ul{Claim 3}. There exists $x_1<x_{3/2}<x_2$ so that $E(x_1,x_{3/2}),E(x_{3/2},x_2)\in\cT$ and 
\begin{equation}\label{partition of B'}
	\cT\cap B'\subseteq  A_3\cup A_4\cup \{T_0\}
\end{equation}
where $A_3, A_4$ are given by Equation (\ref{equation for Aj}).

%
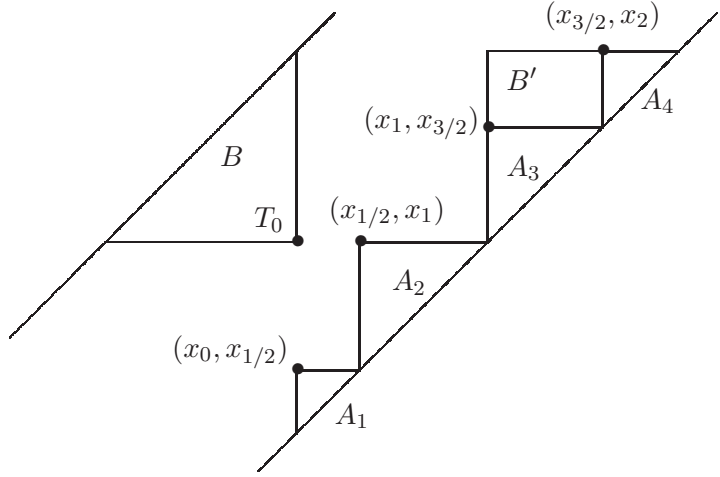
\begin{figure}[htbp]
\begin{center}
%
{
\setlength{\unitlength}{1in}
{\mbox{
\begin{picture}(4,2.6)
    \thinlines
    \put(1.5,1.33){
      \qbezier(-.2,-1.2)(1,0)(2.2,1.2)
      \qbezier(-1.5,-.5)(0,1)(.2,1.2) 
\qbezier(0,0)(0,.5)(0,1)
\qbezier(0,0)(-.5,0)(-1,0)
\qbezier(1,0)(1,.5)(1,1)
\qbezier(1,1)(1.5,1)(2,1)
     \thicklines
\qbezier(.33,0)(.66,0)(1,0)
\qbezier(.33,0)(.33,-.33)(.33,-.66)
\qbezier(0,-.67)(0,-.8)(0,-1)
\qbezier(0,-.67)(.1,-.67)(.33,-.67)
\put(1,1){\qbezier(.6,0)(.8,0)(1,0)
\qbezier(.6,0)(.6,-.33)(.6,-.4)
\qbezier(0,-.4)(0,-.8)(0,-1)
\qbezier(0,-.4)(.1,-.4)(.6,-.4)
}
\put(-.03,-.03){    \put(0,0){$\bullet$}
    \put(0,-.67){$\bullet$}
    \put(.33,0){$\bullet$}
    }
\put(.97,.97){    
    \put(0,-.4){$\bullet$}
    \put(.6,0){$\bullet$}
    }
    \put(-.4,.4){$B$}
    \put(-.22,.08){$T_0$}

    \put(.5,-.25){$A_2$}
    \put(.2,-.95){$A_1$}
    \put(-.65,-.6){($x_0,x_{1/2}$)}
    \put(.17,.13){($x_{1/2},x_1$)}
\put(1,1){        \put(.1,-.2){$B'$}
\put(.8,-.3){$A_4$}
    \put(.1,-.65){$A_3$}
    \put(-.65,-.4){($x_1,x_{3/2}$)}
    \put(.3,.15){($x_{3/2},x_2$)}
      }
      }
\end{picture}}
}}
\caption{$\cT\subseteq A_1\cup A_2\cup A_3\cup A_4\cup\{T_0\}\subseteq A\cup B'$. The points $E(x_j,x_{j+1/2})$ in the figure above must lie in $\cT$ since they are compatible with all points in $A_1\cup A_2\cup A_3\cup A_4\cup\{T_0\}$ and, therefore, with all objects in $\cT$.}
\label{Fig: proof of cluster equivalence}
\end{center}
\end{figure}

We can now construct the mapping
\[
	\f_1:Q_1=\{0,\pi/2,\pi,3\pi/2\}\to S^1
\]
by the formula $\f_1(j\pi)=x_j$ for $j=0,1/2,1,3/2$. Then $\Phi_1(\cT_0\cap \cX_1)$ consists of $T_0$ and the vertices of $A_1,A_2,A_3,A_4$. These lie in $\cT$ by the above claims and the remainder of $\cT$ lies in the union of the triangular regions $A_j$. To apply induction, we use the notation $A_j=A_j^1$ for each $j$.

Suppose that $n\ge 1$ and we have $\f_{n}:Q_{n}\to S^1$ given by $\f_n(j\pi)=x_j$ and we have a sequence of disjoint triangular regions $A_j^n$ so that the image of $\Phi_{n}$ contains the vertex $E(x_{(j-1)/2^n},x_{j/2^n})$ of $A_j^n$ for each $j$ and the remainder $\cT-\Phi_{n}(\cT_0\cap \cX_{n})$ lies in the union of the triangular regions $A_j^n$. Then, in each $A_j^n$, we take the point in $\cT\cap A_j^n$ which is closest to and not equal to the vertex of $A_j^n$. This gives a new point $z=x_{(2j-1)/2^{n+1}}$ in the open interval $(x,y)=(x_{(j-1)/2^n},x_{j/2^n})$ and two new triangular regions $A^{n+1}_{2j-1},A^{n+1}_{2j}$ uniquely determined by their vertices $E(x,z),E(z,y)$, respectively. We also have the following statement analogous to (\ref{partition of A}) and (\ref{partition of B'}) above.
\begin{equation}\label{partition of Anj}
	\cT\cap A_j^n\subseteq A^{n+1}_{2j-1}\cup A^{n+1}_{2j}\cup \{\text{vertex of $A_j^n$}\}
\end{equation}
If we do this for each $j$ we obtain the data that we need to construct $\f_{n+1}$ and we can construct $\f_n$ for all $n$ by induction.

We can use Lemma \ref{second lemma for equivalence of clusters} to finish the proof once we show that every element of $\cT$ is contained in $\Phi_n(\cT_0\cap \cX_n)$ for some $n$. So, let $Z\in\cT$. Let $d>0$ be the distance from $Z$ to the boundary of the Moebius band $\cM$. Let $d_0$ be the distance from $T_0$ to the boundary and let $\e=\min(d,d_0)$. As in the proof of Lemma \ref{second lemma for equivalence of clusters}, there are at most finitely many, say $n$, objects in $\cT$ which have a distance $\ge \e$ from the boundary and $Z, T_0$ are among those $n$ points. (So, $n\ge1$.)

\ul{Claim 4}. $Z$ is contained in $\Phi_{n-1}(\cT_0\cap \cX_{n-1})$.


The case $Z=T_0$ being trivial, we assume that $Z\noteq T_0$ and $n\ge2$. Then $Z\in A^1_j$ for some $j$. Suppose that $A^1_j\cap \cT$ has $m$ elements which are at least $\e$ away from the boundary of $\cM$. Then $n-1\ge m\ge1$ (since $Z$ is one of those $m$ elements but $T_0$ is not). So the claim follows from the following statement which we prove by induction on $m$.

\emph{If $Z\in A^k_j$ and $A^k_j$ contains $m\le n-k$ elements which are at least $\e$ away from the boundary of $\cM$ then $Z\in \Phi_{n-1}(\cT_0\cap \cX_{n-1})$.}

Let $V$ be the vertex of $A^k_j$. Then, since $\e\ge d_0$, $V$ is one of the $m$ points of $A^k_j\cap \cT$ which are at least $\e$ away from the boundary of $\cM$. If $m=1$ then $Z=V$ and $Z\in \Phi_k(\cT_0\cap \cX_k)\subseteq \Phi_{n-1}(\cT_0\cap \cX_{n-1})$ since $k\le n-m\le n-1$. So, suppose that $m>1$ and $Z\neq V$. By (\ref{partition of Anj}), $Z$ is contained in either $A^{k+1}_{2j}$ or $A^{k+1}_{2j-1}$. Suppose, e.g., that $Z\in A^{k+1}_{2j}$. Then $A^{k+1}_{2j}\cap \cT$ will have at most $m-1$ elements which are at least $\e$ away from the boundary of $\cM$ since it is missing $V$. But $m-1\le n-(k+1)$. Therefore, the induction statement holds for all $m$ and the Claim holds for all $n$.

Our theorem follows from Claim 4 by Lemma \ref{second lemma for equivalence of clusters}.
\end{proof}



\section{Mutation of clusters}\label{sec5}


In order to verify that the clusters in $\cC$ form a cluster structure, we need to construct the quiver with potential associated to any cluster. Since clusters are all equivalent, they have isomorphic quivers with potential. Since the objects of a cluster are all equivalent, this quiver has the same structure at every vertex. It resembles a Cayley graph. In fact it is isomorphic to the Cayley graph of the group $\ZZ/3\ast\ZZ/3$ (the free product of the cyclic group of order 3 with itself). (See \eqref{eq: Cayley} below.)

\subsection{The quiver of a cluster}

First we recall that the \emph{quiver} $Q_\cT$ of the category $\cT$ has one vertex $v_i$ for each object $T_i$ in $\cT$ and the number of arrows $v_i\to v_j$ is equal to the dimension of 
\[
	Irr(T_j,T_i)=\Hom(T_j,T_i)/rad(T_j,T_i)
\]
 where $rad(T_j,T_i)$ is the vector subspace of $\Hom(T_j,T_i)$ spanned by all morphisms which factor as $T_j\xrarrow f Z\xrarrow g T_i$ where $f,g$ are not isomorphisms. (This description uses the fact that all objects of $\cT$ are indecomposable. Also, this is the quiver of $\End_\cC(\cT)^{op}$.) 
 
 \begin{lem}
 Suppose that $f:X\to Y$ is a nonzero morphism between compatible nonisomorphic indecomposable objects of $\cC$. If we form a triangle
 \[
 	X\xrarrow f Y\xrarrow g Z\xrarrow h \Sig X
 \]
then $Z$ is indecomposable and compatible with both $X$ and $Y$.
 \end{lem}
 
 \begin{proof} Since $X,Y$ share one end, they each have one free end. Let $Z$ be the unique object with those two ends. Then $Z$ is compatible with $X,Y$ and we have a basic positive or negative triangle $X\to Y\to Z\to \Sig X$. Since the triangle is unique up to isomorphism, it is isomorphic to the given triangle.
 \end{proof}

\begin{prop}\label{prop: local structure of T}
For any object $T$ of any cluster $\cT$ there are nonisomorphic objects $A,B,C,D$ in $\cT$ satisfying the following.
\begin{enumerate}
\item $T$ is the source of exactly two irreducible maps $a:T\to A,f:T\to C$ and the target of exactly two irreducible maps $b:B\to T, g:D\to T$. 
\item There are irreducible morphisms $c:A\to \Sig B\cong B,h:C\to \Sig D\cong D$ so that
$
	B\xrarrow b T\xrarrow a A\xrarrow c \Sig B$ and $D\xrarrow g T\xrarrow f C\xrarrow h \Sig D
$
are triangles:
\[\xymatrixrowsep{10pt}\xymatrixcolsep{10pt}
\xymatrix{
B\ar[dr]_b && A\ar[ll]^c\\
	& T\ar[ur]_a\ar[dl]_f \\
	C\ar[rr]^h && D\ar[ul]_g
	}
\]
\item If $T\to T'$ is a nonzero morphism with $T'\in \cT$ not isomorphic to $A$ or $C$ then the third object $Z$ in the triangle $T\to T'\to Z\to \Sig T$ does not lie in $\cT$.
\end{enumerate}
\end{prop}

\begin{proof}
By the Equivalence of clusters (Theorem \ref{thm: equivalence of clusters}) we may assume that $\cT=\cT_0$ and $T=T_0=E(0,\pi)$. Then all nonzero morphisms between $T_0$ and other objects of $\cT_0$ lie in the linear subcategories $\cP_0,\cP_\pi$. All of the morphisms in $\cP_0$ starting and ending at $T_0$ factor uniquely through the objects in $\cP_0$ which are closest to $T_0$ which are $A=E(0,3\pi/2), D=E(0,\pi/2)$. Similarly, all morphisms in $\cP_\pi$ starting and ending at $T_0$ factor uniquely through $C=E(\pi/2,\pi)$ and $B=E(-\pi/2,\pi)$, respectively. This proves (1).

Since $\Sig D=D'=E(\pi/2,2\pi)$ and $C$ have the same first coordinate they both lie in $\cP_{\pi/2}$. Since $\pi/2<\pi<2\pi$, there is a morphism $C\to \Sig D$. For any angle $\th$ between $\pi$ and $2\pi$, $E(\pi/2,\th)$ is not compatible with $T_0=E(0,\pi)$ since $0<\pi/2<\pi<\th<2\pi$. Therefore, the morphism $C\to \Sig D$ is irreducible. Since the two unshared ends of $C,\Sig D$ are the ends of $T_0$, there is a triangle $D\to T\to C\to \Sig D$. Since $\Hom(D,T)$ and $\Hom(T,C)$ are one dimensional, we may assume that the first two morphisms are $g$ and $f$. Similarly, there is an irreducible morphism $A\to \Sig B$ which generates a triangle $B\xrarrow b T\xrarrow a A\to \Sig B$. This proves (2).

To prove (3) suppose that $T\to T'$ is a nonzero morphism in $\cT$ which is not irreducible. By the Equivalence or clusters we may assume that $\cT=\cT_0$ and $T'=T_0=E(0,\pi)$. By symmetry we may assume $T=E(0,\pi/2^n)$ and $n\ge2$ since $T\to T'$ is not irreducible. But then the third term in the triangle $T\to T'\to Z\to \Sig T$ is $Z=E(\pi/2^n,\pi)$ which is not in $\cT_0$ since $\pi-\pi/2^n$ is not equal to $\pi/2^m$ for any $m$.
\end{proof}

\begin{cor}
The morphism $(b,g):B\oplus D\to T$ is a minimal right $add(\cT\backslash\{T\})$ approximation of $T$ and any morphism $T'\to T$ in $add\,\cT$ with $T'\in add(\cT\backslash\{T\})$ factors uniquely through $(b,g)$. Similarly, $\binom af:T\to A\oplus C$ is a  minimal left $add(\cT\backslash\{T\})$ approximation of $T$. \qed
\end{cor}

\begin{cor}\label{cor: characterization of irreducible morphisms}
Suppose that $k:T_i\to T_j$ is a nonzero nonisomorphism in $\cT$. Then the following are equivalent. 
\begin{enumerate}
\item $k$ is irreducible.
\item The third term $Z$ in the triangle $T_i\to T_j\to Z\to \Sig T_i$ lies in $\cT$.
\item The objects between $T_i$ and $T_j$ in the unique linear subcategory of $\cC$ containing both $T_i$ and $T_j$ do not lie in $\cT$.\qed
\end{enumerate}
\end{cor}

\begin{cor}
In the quiver $Q_\cT$ of a cluster $\cT$, every arrow is contained in exactly one triangle. The composition of any two arrows in each triangle is zero. Each vertex lies on two triangles which are otherwise disjoint.\qed
\end{cor}

To describe the quiver of a cluster, we use group theory. Let $G$ be the free product of two copies of $\ZZ/3$ with generators and relations:
\[
	G=\ZZ/3\ast\ZZ/3=\<a,b:a^3,b^3\>
\]
Recall that the \emph{Cayley graph} $C(G)$ of a group $G$ with $n$ generators, none of order 2, is a directed graph with one vertex for every element of $G$ and $2n$ edges at each vertex, $n$ pointing outward and labeled with the generators and $n$ pointing inward and labeled with the $n$ generators so that for any arrow labeled $a: x\to y$ in the graph the source and target are related by $y=xa$ in the group. This graph is the 2-skeleton of a 2-dimensional cell complex $K(G)$ in which, at each vertex we attach a 2-cell (a polygon) for every relation. $K(G)$ is always simply connected and has a free action of $G$ on the left. In this particular example, the 2-cells are triangles and $K(G)$ is contractible.

\begin{equation}\label{eq: Cayley}
\xymatrixrowsep{10pt}\xymatrixcolsep{10pt}
\xymatrix{
&&\cdots&\,\ar[ld]&&&&&\cdots&\,\ar[ld]\\
&&\bullet\ar[rdd]^b\ar[lu] &&&&&& \bullet \ar[lu]\ar[ddr]^a&&\\
\,\ar[dr]&&&&&&&&&&\,\\
\cdots&\bullet\ar[ruu]^b\ar[dl] &&\bullet\ar[ll]^b\ar[rrd]^a&&&& \bullet\ar[ruu]^a\ar[dd] &&\bullet\ar[ll]^a\ar[ru]&\cdots\\
\,&&&&&\bullet\ar[dll]^a\ar[urr]^b &&&&&\,\ar[ul]\\
\cdots\ar[r]&\bullet\ar[rr]^b\ar[ld]&&\bullet\ar[ldd]^b\ar[uu]^a&&&& \bullet\ar[llu]^b\ar[rr]^a &&\bullet\ar[ldd]^a\ar[r]&\cdots\\
\,&&&&&&&&&&\,\ar[lu]\\
&&\bullet\ar[luu]^b\ar[rd]&&&&&& \bullet \ar[uul]^a\ar[rd]&&\\
&\,\ar[ru]&\cdots&&&&&\,\ar[ru]&\cdots&&
	}
\end{equation}

\begin{thm}
The quiver $Q_\cT$ of any cluster $\cT$ is isomorphic to the Cayley graph $C(G)$ of the group $G=\ZZ/3\ast\ZZ/3$ with generators and relations given as above.
\end{thm}

\begin{proof}
Let $\cT=\cT_0$ then an isomorphism $\psi: C(G)\cong Q_{\cT_0}$ is given as follows. The vertices of $C(G)$ are labeled with elements of $G$ which are given by finite sequences
\[
	g=a^{\e_1}b^{\e_2}a^{\e_3}\cdots x^{\e_n}\quad\text{or}\quad g= b^{\e_1}a^{\e_2}b^{\e_3}\cdots y^{\e_n}
\]
where each $\e_i$ is 1 or $-1$ and $x=a$ or $b$ depending on the parity of $n$ and similarly for $y$. We map the identity $e\in G$ to $T_0=E(0,\pi)$ and other elements of $G$ to 
\[
	\psi(a^{\e_1}b^{\e_2}a^{\e_3}\cdots x^{\e_n})=T_0+\sum_{j=1}^n \frac{\pi}{2^{j+1}}[(-1,1)-\e_j(1,1)]
\]
\[
	\psi(b^{\e_1}a^{\e_2}b^{\e_3}\cdots y^{\e_n})=T_0+\sum_{j=1}^n \frac{\pi}{2^{j+1}}[(1,-1)-\e_j(1,1)]
\]
This gives a bijection between the elements of $G$ and the objects of $\cT_0$ and triangles are sent to triangles.
\end{proof}

Let $W$ be the potential on $Q_\cT$ given by the sum of all triangles: $W=\sum a_ib_ic_i$. Then the \emph{Jacobian category} $\cJ(Q_\cT,W)$ is given by the quiver $Q_\cT$ modulo the relations given by taking $\d_aW=0$ where $\d_aW$ is the cyclic derivative of $W$ with respect to the arrow $a$. In this case these relations are that the composition of any two arrows in any triangle is zero. It is clear that the nonzero paths in this quiver are the paths which have at most one arrow in every triangle. These are exactly the paths contained in the linear subcategories (but going the wrong way). This gives the following.

\begin{thm}
$\cJ(Q_\cT,W)\cong \cT^{op}$.
\end{thm}

\subsection{Mutation} We show that the collection of all clusters in $\cC$ forms a cluster structure as defined in \cite{BIRSc} and \cite{BIRSm}.

\begin{prop}\label{prop: mutation by octahedral axiom}
(a) Given any object $T$ in any cluster $\cT$ in $\cC$ then, up to isomorphism, there is a unique object $T^\ast$ in $\cC$ not isomorphic to $T$ so that $\cT^\ast=\cT\backslash \{T\}\cup \{T^\ast\}$ is a cluster.

(b) This new object $T^\ast$ is given by the octahedral axiom for triangulated categories where the objects, morphisms and triangles on the left are as given in Proposition \ref{prop: local structure of T}.
\[\xymatrixrowsep{10pt}\xymatrixcolsep{10pt}
\xymatrix{
B\ar[dr]_b && A\ar[ll] &&&& B\ar[dd]_{fb} && A\ar[dl]\\
	& T\ar[ur]_a\ar[dl]_f &&&\longrightarrow&&& T^\ast\ar[ul]\ar[dr]\\
	C\ar[rr] && D\ar[ul]_g&&&& C\ar[ur]&&D\ar[uu]_{ag}
	}
\]

(c) In particular (as part of the octahedral axiom), we have triangles $T\to A\oplus C\to T^\ast\to \Sig T$ and $T^\ast \to B\oplus D\to T\to \Sig T^\ast$ and these triangles are $add(\cT\backslash\{T\})$ approximations of $T,T^\ast$ on the left and right.

(d) All morphisms in the right hand diagram above are irreducible morphisms in $\cT^\ast$.
\end{prop}

\begin{proof} Take $\cT=\cT_0$ and $T=T_0=E(0,\pi)$. Let $T_0^\ast=E(\pi/2,3\pi/2)$. Then $\cT_0^\ast=\cT_0\backslash\{T_0\}\cup\{T_0^\ast\}$ is a cluster since it is equal to $\Phi_\f(\cT_0)$ where $\f:S^1\to S^1$ is rotation by $\pi/2$. The problem is to show that $T_0,T_0^\ast$ are the only objects which will complete $\cT_0\backslash\{T_0\}$ into a cluster.
 
To prove this we note that, when $T$ is deleted from $\cT$, the morphism $fb:B\to C$ becomes irreducible. This follows from Corollary \ref{cor: characterization of irreducible morphisms} ({characterization of irreducible morphisms}) since $T$ is the unique point between $B$ and $C$ in the linear subcategory containing $B$ and $C$ which is compatible with all objects in $\cT\backslash\{T\}$. The reason is that all of these points are also compatible with $T$, so they should already be in $\cT$. Therefore, if we replace $T$ with say $Y\ncong T$ then the morphism $B\to C$ becomes irreducible. By Corollary \ref{cor: characterization of irreducible morphisms}, this implies that the third terms $T_0^\ast$ in the triangle $B\to C\to T_0^\ast\to \Sig B$ must be in the cluster. And this implies that $Y\cong T_0^\ast$ as claimed. The other statements follow from Proposition \ref{prop: local structure of T}.
\end{proof}

\begin{cor}
Let $W^\ast$ be the potential for the quiver $Q_{\cT^\ast}$ of the mutation $\cT^\ast$ of $\cT$ at $T$. Then $(Q_{\cT^\ast},W^\ast)$ is obtained from $(Q_\cT,W)$ by Derksen-Weyman-Zelevinsky (DWZ) mutation of quivers with potential \cite{DWZ1}.
\end{cor}

\begin{proof}
We could carry out the calculation. Or we could simply quote known results. The part of the quiver $Q_\cT$ drawn in the above proposition is the quiver with potential for a cluster mutation between clusters in the cluster category of type $A_5$ and we know by \cite{DWZ1} and \cite{BIRSm} that this is given by DWZ mutation. Note that the definition of triangulation in $\cC$ and $\cC_{A_5}$ do not agree.
\end{proof}

This proposition and corollary prove the following.

\begin{thm}\label{C has cluster structure}
The triangulated category $\cC=\cC_\pi$ has a cluster structure in the sense of \cite{BIRSc} and \cite{BIRSm} with clusters defined in Section \ref{sec3} above.\qed
\end{thm}

\subsection{Rational cluster category}

Given a cluster $\cT$ we say that an indecomposable object $E(x,y)\in\cC$ is \emph{rational} with respect to $\cT$ if both $x$ and $y$ are ends of objects in $\cT$, i.e., if and only if $\cT$ has objects in $\cP_x$ and in $\cP_y$. An arbitrary element of $\cC$ is called rational with respect to $\cT$ if it is a direct sum of rational objects. Thus $\cX$ is the triangulated full subcategory of $\cC$ of all objects rational with respect to the standard cluster $\cT_0$.

\begin{thm}
For any cluster $\cT$, the union of all clusters obtainable from $\cT$ by a finite sequence of mutations is the set of all indecomposable objects which are rational with respect to $\cT$.
\end{thm}

\begin{proof}
It is clear that all objects in any iterated mutation of $\cT$ are rational. We need to prove the converse. We may assume that $\cT=\cT_0$. Take any rational objects, say $T=E(x,y)$. Then $x=a\pi/2^n,y=b\pi/2^m$. So, $T$ is compatible with all objects of $\cT_0$ of the form $E(c\pi/2^k,(c+1)\pi/2^k)$ for $k\ge n,m$. Therefore, there are only finitely many objects in $\cT_0$ which are not compatible with $T$. If this number is zero then $T\in\cT_0$. If the number is greater than zero, then there is some object $E(a,b)=E(c\pi/2^k,(c+1)\pi/2^k)$ which is not compatible with $T$. Choose $k$ to be maximal and let $c=(a+b)/2=(2c+1)\pi/2^{k+1}$. Then by maximality of $k$, $E(a,c)$ and $E(c,b)$ are compatible with $T$ which means that either $x=r$ or $y=r$. If we mutate $\cT_0$ at vertex $E(a,b)$ then we replace $E(a,b)$ with $E(c,d)$ for some $d$. Therefore, the new cluster $\cT_0$ has fewer objects which are not compatible with $T$. So, by repeating this process, we can eliminate all objects which cross $T$ and $T$ will be in the mutated cluster.
\end{proof}

\subsection{Clusters in $\cC_{r,s}$ for $r<s$}


\begin{thm}\label{cluster structures only exist for certain c} 
For any $0<r<s$ so that the ratio $\frac rs$ is not of the form $\frac rs=\frac{n+1}{n+3}$, there do not exist any weak cluster structures on the triangulated orbit category $\cC_{r,s}$. 
\end{thm}

Before we prove this we examine mutations in any possible weak cluster structure on $\cC_{r,s}$ without coefficients when $r<s$. Recall that $\cC_{r,s}\cong \cC_c$ where $c=\frac{r\pi}s$.

Suppose that $\cC_c$ has a weak cluster structure, $\cT$ is a cluster in $\cC_c$ and $T\in\cT$. By assumption we have triangles $T\to B\to T^\ast\to \Sig T$ and $\Sig^{-1}T\to T^\ast\to B'\to T$ in $\cF_c$ in which $B,B'$ are the left and right $add(\cT\backslash T)$-approximations for $T$ and $\cT\backslash T\cup T^\ast$ is another cluster. We have an exact sequence in the Frobenius category $\cF_c$:
\[
	T\to B\oplus IT\to T^\ast\oplus J
\]
where $IT=I_1T\oplus I_2T$ is the injective envelope of $T$ in $\ccc_r$ and $J$ is a summand of $IT$.

If we split off the sequence $0\to J\to J$, we get another exact sequence:
\[
	T\to B\oplus IT/J\to T^\ast
\]
where the middle term must have exactly two components since, otherwise, $T,T^\ast$ would be compatible. If $T=E(a,b)$ then this exact sequence must have the following form where $T^\ast\cong E(a^+,b^+)$.
\begin{equation}\label{left approximation of T}
	E(a,b)\to E(a^+,b)\oplus E(a,b^+)\to E(a^+,b^+)
\end{equation}
Similarly, we have the dual approximation sequence
\begin{equation}\label{right approximation of T}
	E(a^-,b^-)\to E(a^-,b)\oplus E(a,b^-)\to E(a,b)
\end{equation}
where $E(a^-,b^-)\cong T^\ast\cong E(a^+,b^+)$. But $b-2\pi+\th\le a^-<a<a^+\le b-\th$ which means that $a^+,a^-$ are not the same point on the circle $S^1$. So, we have $a^-\equiv b^+$ and $a^+\equiv b^-$ modulo $2\pi$. But $a+\th\le b^-<b$. So, we must have:
\begin{equation}\label{equation from two approximations}
	a+\th\le b^-=a^+\le b-\th.
\end{equation}
The theorem now follows from the next lemma.

\begin{lem}\label{a-b is integer multiple of theta}
Suppose that $c<\pi$ and $\cC_c$ has a weak cluster structure without coefficients. Suppose that $\cT$ is a cluster in $\cC_c$ and $T\cong E(a,b)\in\cT$ then $b-a$ is an integer multiple of $\th=\pi-c$.
\end{lem}

\begin{proof}[Proof of Theorem \ref{cluster structures only exist for certain c}]
Let $E(a,b)$ be an object of any cluster in $\cC_c$. Then $a+\th<b$. So, by the Lemma, we must have $b-a\ge 2\th$. Also, $E(a,b)\cong E(b,a+2\pi)$. Therefore, $a+2\pi-b$ is also an integer multiple of $\th$ and at least equal to $2\th$. Adding these we see that $2\pi=N\th$ where $N\ge4$. So, $N=n+3$ for some positive integer $n$. Then, $\th=2\pi/N=2\pi/(n+3)$ which implies that $c=\pi-\th=(n+1)\pi/(n+3)$ as claimed.
\end{proof}

\begin{proof}[Proof of Lemma]
Suppose not.
Let $m$ be the smallest positive integer so that
\[
	m\th<b-a<(m+1)\th
\]
for some $E(a,b)\cong T\in\cT$. Then the left and right $add(\cT\backslash T)$ approximations of $T$ gives us triangles (\ref{left approximation of T}), (\ref{right approximation of T}) satisfying (\ref{equation from two approximations}).

\ul{Claim} $b-a^+$ is an integer multiple of $\th$. 

\emph{Pf:} Either $a^+=b-\th$ or $a^+<b-\th$. In the first case, the claim holds. In the second case, $E(a^+,b)$ is an object in the cluster $\cT$. But $b-a^+\le b-a-\th$ (since $a^+\ge a+\th$). By minimality of $m$, we must have that $b-a^+=k\th$ for some integer $k$ as claimed.

The dual argument shows that $b^--a$ is also an integer multiple of $\th$. Since $a^+=b^-$ we conclude that $b-a=(b-a^+)-(b^--a)$ is an integer multiple of $\th$ as claimed.
\end{proof}

In the case when $c=(n+1)\pi/(n+3)$ we will show that any cluster structure on $\cC_c$ is the one we already know (given by embeddings of the cluster category of type $A_n$ into $\cC_c$).

\begin{thm}
Suppose that $c=(n+1)\pi/(n+3)$ and we have a weak cluster structure on $\cC_c$. Then every cluster has exactly $n$ objects and for any two objects $T_i=E(a,b),T_j=E(x,y)$ in any cluster $\cT$ the real numbers $a,b,x,y$ differ by integer multiples of $\th=\pi-c$. Furthermore, $\Ext^1(T_i,T_j)=\Ext^1(T_j,T_i)=0$. In other words, $\cT$ is the image of a cluster in the cluster category of type $A_n$ under an embedding into $\cC_c$.
\end{thm}

\begin{proof}
$T_i=E(a,b)\cong E(b,a+2\pi)$ and $\Sig T_i\cong E(b-\th,a+2\pi-\th)$. Any indecomposable object $Z$ of $\cC_c$ with $\Ext^1(Z,T_i)=\Hom(Z,\Sig T_i)\neq0$ is isomorphic to $E(x,y)$ where $(x,y)$ lies in the half-open rectangle $(a,b-\th]\times (b,a+2\pi-\th]$. Therefore, the statement that $\Ext^1(T_j,T_i)=0$ for all $T\in\cT$ will follow from the following claim.\vs2

\ul{Claim 1} $\cT$ contains no object $E(x,y)$ with $(x,y)$ in the open rectangle $(a,b)\times (b,a+2\pi)$.\vs2

\emph{Pf:} Suppose that there is an $E(x,y)$ in the open rectangle. Let $(m_1,m_2)$ be minimal in lexicographic order among all counterexamples to Claim 1 where $m_1$ is the minimum and $m_2$ is the maximum of
\[
	\left[\frac{x-a}\th
	\right]\quad\quad \left[\frac{y-b}\th
	\right]
\]
where $[\ ]$ means integer part. Consider the left $add(\cT\backslash T_i)$-approximation sequence (\ref{left approximation of T}). This sequence must also be the right $add(\cT\backslash T_i)$-approximation sequence for $T_i^\ast=E(a^+,b^+)$. Therefore, the point $(x,y)$ cannot lie in the closed rectangle $[a,a^+]\times [b,b^+]$. This implies that $(x,y)\in R_1\cup R_2$ where $R_1=(a^+,b)\times(b,a+2\pi)$ and $R_2=(a,b)\times (b^+,a+2\pi)$.

Case 1. If both objects $E(a,b^+),E(a^+,b)$ are nonzero in $cC_r$ then the inequalities 
\[
	(x+y)-(a^++b)\le (x+y)-(a+b)-\th
\]
\[
	(x+y)-(a+b^+)\le (x+y)-(a+b)-\th
\]
give a contradiction to the minimality of $m$. Therefore, either $a^+=b-\th$ or $b^+=a+2\pi-\th$.

Case 2. $a^+=b-\th$ ($E(a^+,b)=0$) but $E(a,b^+)\neq0$ in $\cC_{r,s}$. In this case $(x,y)$ is not in $R_2$, so it must be in $R_1$. So, $b-\th<x<b$ and $(a,b)$ is in the forbidden open rectangle $(x,y)\times (y,x+2\pi)$ for the object $E(x,y)\in\cT$. For this counterexample we have $m_1=0$ since $b-x<\th$. Therefore, $m_1=0$. This means either $x<a+\th$ or $y<b+\th$. The second case is impossible since, by Lemma \ref{a-b is integer multiple of theta}, $y\ge x+2\th$ and $x>b-\th$. Therefore, $x<a+\th< b$ which implies that $(x,y)$ cannot lie in $R_1$. So, Case 2 is not possible. 

A similar argument shows that the other two cases are also not possible and Claim 1 is established.\vs2

Claim 1 can also be rephrased to say that, for any two objects $E(a,b),E(x,y)$ in a cluster, the pairs $\{a,b\},\{x,y\}$ are \emph{noncrossing}. So, we can choose representatives so that $a\le x<y\le b$ (up to reordering of $a,b$ and $x,y$).
\vs2

\ul{Claim 2} For any two objects $T_i=E(a,b),T_j=E(x,y)$ in $\cT$, the real numbers $a,b,x,y$ differ by integer multiples of $\th$.\vs2

\emph{Pf:} Let $b-a=m\th$ be minimal among all counterexamples so that $a\le x<y\le b$. Then we note that $b-a\ge 3\th$ since, if $b-a=2\th$, we would have $a=x$ and $b=y$. Therefore, in the inequality (\ref{equation from two approximations}) we must have either $a+\th<b^-$ or $a^+<b-\th$. In the first case $E(a,b^-)\in\cT$ and, by Claim 1, either $a\le x<y\le b^-$ or $b^-\le x<y\le b$ both of which are impossible either by induction on $m$ or by the fact that $y-x\ge2\th$. The other case is similar which shows that there are no counterexamples to Claim 2.\vs2

By what we have proved so far, all objects $E(x,y)$ in the cluster $\cT$ have endpoints $x,y$ in the set
\[
Z=\{a+m\th\,|\,m=0,\cdots,n+2\}
\]
Furthermore, $|x-y|\ge 2\th$, so each object corresponds to an internal chord in the regular $(n+3)$-gon with vertex set $Z$. We have shown that compatibility of such objects is equivalent to the condition that the corresponding internal chords do not cross. Thus $\cT$ is a cluster in the cluster category $\cF(Z)$ which is equivalent to the cluster category of type $A_n$ over $\kk$ by Theorem \ref{F(Z)=C(An) is embedded in Cc}. In particular, $\cT$ has exactly $n$ objects.
\end{proof}

\begin{cor}\label{cluster structures for Cc}
The continuous orbit category $\cC_c$ has a cluster structure without coefficients if and only if either $c=\pi$ or $c=(n+1)\pi/(n+3)$ for some nonnegative integer $n$.
\end{cor}

\begin{proof}
By the theorem, the restrictions on $c$ are necessary. But we showed in Theorem \ref{F(Z)=C(An) is embedded in Cc}, that for $c=(n+1)\pi/(n+3)$, there is a triangulated embedding of the cluster category $\cC_{A_n}$ into $\cC_c$ given by choosing $n+3$ equally spaced points on the circle. The image in $\cC_c$ of the clusters in the cluster category $\cC_{A_n}$ give a cluster structure on $\cC_c$ and the theorem above implies that unions of such collections of clusters are the only possible cluster structure on $\cC_c$ without coefficients.
\end{proof}



\end{document}